\documentclass[letterpaper,11pt,reqno]{amsart}
\usepackage[margin=1.19in]{geometry}
\usepackage{amsmath,amsthm,amssymb,mathtools,amsfonts,mathrsfs}
\usepackage{xspace,xcolor}
\usepackage[breaklinks,colorlinks,citecolor=teal,linkcolor=teal,urlcolor=teal,pagebackref,hyperindex]{hyperref}
\usepackage[alphabetic]{amsrefs}
\usepackage{tikz-cd}
\usepackage[all]{xy}
\usepackage{bbm}
\usepackage{array}
\usepackage{setspace}
\usepackage{makecell}
\usepackage{diagbox}

\newtheorem{thm}{Theorem}[section]

\newtheorem{cor}[thm]{Corollary}

\theoremstyle{definition}

\newtheorem{rmk}[thm]{Remark}

\newtheorem*{nota*}{Notation}

\newcommand{\tabincell}[2]{\begin{tabular}{@{}#1@{}}#2\end{tabular}} 
\newcommand{\bM}{\overline{\mathcal M}}

\newcommand{\g}{\text{g}}

\newcommand{\vir}{{\operatorname{vir}}}

\newcommand\C{\mathbb C}

\newcommand\N{\mathbb N}
\newcommand\PP{\mathbb P}

\newcommand\Q{\mathbb Q}
\newcommand\Z{\mathbb Z}

\newcommand\cO{\mathcal O}

\newcommand\Aut{\operatorname{Aut}}

\newcommand\Spec{\operatorname{Spec}\,}

\newcommand\Hom{\operatorname{Hom}}

\newcommand\Ker{\operatorname{Ker}\,}

\newcommand\Rep{\operatorname{Rep}}

\newcommand\sG{\mathcal{G}}

\newcommand\Tg{T_{\log}}
\newcommand\lne{\mathrm{line}}
\newcommand\conic{\mathrm{conic}}

\newcommand\bard{\underline{d}}
\newcommand\Exp{\mathrm{Exp}}

\title[All-genus WDVV recursion, quivers, and BPS invariants]{All-genus WDVV recursion, quivers, and BPS invariants}

\author{Pierrick Bousseau}
\address{University of Georgia, Department of Mathematics, Athens, GA 30605, US}
\email{Pierrick.Bousseau@uga.edu}
\author{Longting Wu}
\address{Southern University of Science and Technology, Department of Mathematics \& SUSTech International Center for Mathematics, Shenzhen, Guangdong 518055, CN}
\email{wult@sustech.edu.cn}

\begin{document}

\begin{abstract}
Let $X$ be a smooth projective surface and $D$ a smooth rational ample divisor in $X$.
We prove an all-genus generalization of the genus $0$ WDVV equation for primary Gromov--Witten invariants of the local 3-fold $\cO_X(-D)$.
The proof relies on a correspondence between all-genus Gromov--Witten invariants and refined Donaldson--Thomas invariants of acyclic quivers. In particular, the corresponding BPS invariants are expressed in terms of Betti numbers of moduli spaces of quiver representations.
\end{abstract}

\maketitle

\section{Introduction}
\subsection{All-genus recursion}
Let $X$ be a smooth projective surface over $\mathbb{C}$ and $D$ be an ample divisor in $X$. Let $\beta\in H_2(X,\Z)$ be a nonzero curve class of $X$. We use $\bM_{g,m}(\mathcal{O}_X(-D),\beta)$ to denote the moduli space of $m$-pointed genus $g$ stable maps of class $\beta$ to the total space of the line bundle $\mathcal{O}_X(-D)$. Since $D$ is ample, we have $D\cdot\beta>0$, and so $\bM_{g,m}(\mathcal{O}_X(-D),\beta)$ coincides with the moduli space $\bM_{g,m}(X,\beta)$ of $m$-pointed genus $g$ stable maps to $X$ with class $\beta$. However, the Gromov--Witten virtual class $[\bM_{g,m}(\mathcal{O}_X(-D),\beta)]^{\vir}$ is in general different from $[\bM_{g,m}(X,\beta)]^{\vir}$.
We consider the following primary Gromov--Witten invariant:
\[N_{g,\beta}^{\cO_X(-D)}\coloneqq \int_{[\bM_{g,m}(\mathcal{O}_X(-D),\beta)]^{\vir}}\prod_{i=1}^m ev_i^*([pt])\,,\]
where $[pt] \in H^4(X)$ is the point class of $X$, and $ev_i : \bM_{g,m}(\mathcal{O}_X(-D),\beta) \rightarrow X$ is the evaluation at the $i$-th marked point.
The virtual dimension of $\bM_{g,m}(\mathcal{O}_X(-D),\beta)$ is $\Tg\cdot \beta +m$, where $\Tg=-K_X-D$, and so we need $m=\Tg\cdot \beta$
in order to have a possibly nonzero invariant. In particular, we need $\Tg\cdot \beta\geq 0$.
By fixing $\beta$ and summing over $g$, we get the following generating series
\[F_{\beta}^{\cO_X(-D)}\coloneqq \sum_{g\geq 0} N_{g,\beta}^{\cO_X(-D)}h^{2g-2+\Tg\cdot\beta}.\]
By the Gopakumar--Vafa conjecture proven in \cites{Z,IP2,DW19,DIW}, we know that $F_{\beta}^{\cO_X(-D)}\in \Q((-q)^{-\frac{1}{2}})$, i.e.\ a rational function of $(-q)^{-\frac{1}{2}}$. Here $q$ is related to $h$ by $q=e^{\sqrt{-1}h}$. But these $F_{\beta}^{\cO_X(-D)}$ are quite hard to compute in general. In this paper, we give a simple uniform recursive formula for $F_{\beta}^{\cO_X(-D)}$ when $D$ is further assumed to be of virtual genus 0. Here the \emph{virtual genus} $g(D)$ of a divisor $D$ is defined by
$g(D)\coloneqq 1-\frac{1}{2}\Tg\cdot D.$
\begin{thm}\label{thm:main}
Let $X$ be a smooth projective surface over $\C$ and $D$ be an ample divisor in $X$ with virtual genus 0. Then, we have the following recursive formula:
\begin{equation}\label{eqn:recurfm}
F_{\beta}^{\cO_X(-D)}=\sum_{\substack{\beta_1+\beta_2=\beta\\ \beta_1,\beta_2>0}} F_{\beta_1}^{\cO_X(-D)}F_{\beta_2}^{\cO_X(-D)}\left(q^{D\cdot\beta_1}+q^{-D\cdot\beta_1}-2\right){\Tg\cdot\beta-3\choose \Tg\cdot\beta_1-1}
\end{equation}
if $\Tg\cdot\beta\geq 3$.
Here ${*\choose *}$ stand for the binomial coefficients.
\end{thm}
\begin{rmk}
We use the convention that 
\[{m\choose n}=0,\quad \text{if\,\,}n<0\text{\,\,or\,\,}n>m\,.\]
So only the classes $\beta_i$ ($i=1,2$) with $\Tg\cdot\beta_i>0$ are involved in the summation on the right-hand side of \eqref{eqn:recurfm}.
\end{rmk}

\begin{rmk}
The recursion formula might not hold if we relax the ample condition and only require $D$ to be nef. 
We give a counterexample in Appendix \ref{sec:ex}.
\end{rmk}


According to \cite[Corollary 2.3]{LP}, the condition that $X$ has an ample divisor $D$ with virtual genus $0$ actually forces $(X,D)$ to be the following two types:
\begin{enumerate}
\item $(X,D)=(\PP^2,\lne)\text{\,or\,} (\PP^2,\conic)$;
\item $X$ is a Hirzebruch surface and $D \cdot f=1$
for every fiber $f$ of the Hirzebruch surface $X$.
\end{enumerate}
In particular, we can always choose $D$ to be effective with $D \simeq \PP^1$. By taking the lowest coefficient of $h$ in the above recursion, we obtain the following recursion for genus $0$ invariants:
\begin{cor}\label{cor:g0rec}
Under the same assumptions as in Theorem \ref{thm:main}, we have 
\begin{equation}\label{eqn:recurfmg0}
 N_{0,\beta}^{\cO_X(-D)}=-\sum_{\substack{\beta_1+\beta_2=\beta \\
 \beta_1,\beta_2>0}} N_{0,\beta_1}^{\cO_X(-D)}N_{0,\beta_2}^{\cO_X(-D)}\left(D\cdot\beta_1\right)^2{\Tg\cdot\beta-3\choose \Tg\cdot\beta_1-1}
\end{equation}
if $\Tg\cdot\beta\geq 3$.
\end{cor}
At the moment, a direct proof of the all-genus recursion \eqref{eqn:recurfm} in Gromov--Witten theory is not known. Instead, we will first translate the recursion into a recursion for quiver Donaldson--Thomas (DT) invariants via the local/relative correspondence \cites{GGR,BFGW} together with a GW/quiver correspondence derived in \cite{Bou21}
from the GW/Kronecker correspondence for log Calabi-Yau surfaces
\cites{GrP,RSW,RW13,RW21,Bou20}. The recursion on the quiver DT-side can then be deduced using the geometric properties of the quiver moduli spaces, in a way parallel to the work of Reineke--Weist in \cite{RW21}.

But the genus-zero recursion can be directly deduced on the Gromov--Witten side. Actually, it follows from
from the corresponding WDVV recursion for genus-zero relative Gromov--Witten invariants of the pair $(X,D)$ in \cite{FW} plus the local/relative correspondence in \cite{GGR}. The details of such a proof can be found in Appendix \ref{sec:recg0}. It will be clear from the proof that the requirement of $D$ to be rational is necessary. It follows that one can view \eqref{eqn:recurfm} as an example of all-genus WDVV equation.

\begin{rmk}
Another structure underlying higher genus Gromov--Witten invariants is given by the conjectural Virasoro constraints \cite{EHX, EJX}. In genus $0$, the Virasoro constraints are known to essentially follow from the WDVV equation \cite{LTvir}, and so one can also view the Virasoro constraints as a higher genus generalization of the WDVV equation. There is however a major difference between the recursion \eqref{eqn:recurfm} and the Virasoro constraints: the Virasoro constraints are naturally formulated genus by genus in terms of generating series summing over curve classes, whereas the recursion \eqref{eqn:recurfm} is formulated curve class by curve class in terms of generating series summing over the genus.
Actually, as $D$ is ample, the Gromov--Witten invariants of $\cO_X(-D)$ in Theorem \ref{thm:main} can be viewed as Gromov--Witten invariants of the projective compactification $\PP(\cO_X  \oplus \cO_X(-D))$, which is a projective toric variety
and so for which the Virasoro constraints are known \cite{Giv5, viraosorotoric, T12}. In Appendix \ref{app:virasoro}, we compare the recursions for genus one invariants 
derived from \eqref{eqn:recurfm} and from the Virasoro constraint: the two recursions are different and the fact that they produce the same invariants seems highly non-trivial.
\end{rmk}

\begin{rmk}
Recently, there has been an extensive study \cites{BFGW, CI, LhoP, FRZZ, Lho, Wan2} of Gromov-Witten invariants of local Calabi-Yau threefolds, i.e.\,, $\cO_X(-D)$ with $D$ be an anticanonical divisor. One of the key features in their studies is the so-called holomorphic anomaly equation which constrains higher genus Gromov-Witten invariants of local Calabi-Yau threefolds recursively. While in our Fano-like cases, the new feature is the above all-genus WDVV recursion which seems to replace the role of holomorphic anomaly equation. It will be interesting to explore such kinds of all-genus recursions for other Fano threefolds.
\end{rmk}



\subsection{Local/Relative correspondence}
As a first step towards the recursion \eqref{eqn:recurfm}, we relate local Gromov--Witten invariants of $\mathcal{O}_X(-D)$ to relative Gromov--Witten invariants of the pair $(X,D)$. 
Let $\bM_{g,m}(X/D,\beta)$ be the moduli space of $m$-pointed genus $g$ relative stable maps of class $\beta$ to $(X,D)$ with only one contact condition of maximal tangency along $D$. We consider the following maximal contact relative Gromov--Witten invariants:
\[N_{g,\beta}^{X/D}\coloneqq \int_{[\bM_{g,m}(X/D,\beta)]^{\vir}}\prod_{i=1}^m ev_i^*([pt])(-1)^g \lambda_g\]
where $\lambda_g$ is the top Chern class of the Hodge bundle. Similar to the local case, we also need $m=\Tg\cdot\beta$ by the dimension constraint to have possible nonzero invariants. By fixing $\beta$ and summing over $g$, we also form the following generating series:
\[F_{\beta}^{X/D}\coloneqq \sum_{g\geq 0} N_{g,\beta}^{X/D}h^{2g-1+\Tg\cdot\beta}.\]
It is related to $F_{\beta}^{\cO_X(-D)}$ by the following theorem:
\begin{thm}\label{thm:loc/rel}
Under the same assumptions as in Theorem \ref{thm:main} for the pair $(X,D)$, we have 
\[F_{\beta}^{\cO_X(-D)}=F_{\beta}^{X/D}\frac{(-1)^{D\cdot\beta-1}}{2\sin(\frac{(D\cdot\beta)h}{2})}\,.\]
\end{thm}
The above theorem is proven in Section \ref{sec:loc/rel} using the local/relative correspondence \cites{GGR,BFGW}. The generating series $F_{\beta}^{X/D}$ can be further related to Poincar\'e polynomials of certain quiver moduli by the work of first author \cite{Bou21}.

\subsection{Deformation equivalence}\label{subsec:definv}
Before we continue, we give a more detailed description for those pairs $(X,D)$ such that $X$ is a Hirzebruch surface and $D \cdot f=1$. 

We have $X=F_n=\mathbb{P}(\mathcal{O}(n)\oplus \mathcal{O})$ for some $n\in \mathbb{Z}_{\geq 0}$. Let $C_n$ and $C_{-n}$ be the sections of $X$ with intersection numbers $n$ and $-n$ respectively. The requirements that $D \cdot f=1$ and $D$ is ample imply that $D=C_n+sf$ for some $s>0$.

A deformation of $F_{n}$ to $F_{n+2}$ is given by
\[\left\{\left([x_0:x_1],[y_0:y_1:y_2],t\right)\in \PP^1\times\PP^2\times \C\Big|x_0^{n+2}y_1-x_1^{n+2}y_0+tx_0^{n+1}x_1y_2=0\right\}\,.\]
Under such a deformation the divisor $C_n+(s+1)f$ of $F_n$ deforms to the divisor $C_{n+2}+sf$ of $F_{n+2}$, and the curve class $d_1C_{-n}+d_2f$ deforms to the curve class 
$d_1C_{-n-2}+(d_1+d_2)f$. So the intersection numbers $\Tg\cdot\beta$ and $D\cdot\beta$ appearing in the recursion \eqref{eqn:recurfm} do not change. This implies that the  recursion \eqref{eqn:recurfm} is compatible with the above deformation. So after a sequence of deformation, we have
\[\begin{aligned}
\text{GW}(\mathcal{O}_{F_n}(-C_n-sf))&\simeq \text{GW}(\mathcal{O}_{F_{n+2}}(-C_{n+2}-(s-1)f))\simeq\\
\cdots & \simeq  \text{GW}(\mathcal{O}_{F_{n+2s-2}}(-C_{n+2s-2}-f))
\end{aligned}\] 
by the deformation invariance of Gromov--Witten theory.
Thus, it is enough to consider only the following three types of pairs $(X,D)$:
\begin{equation}\label{eqn:3types}
(\PP^2,\lne),\quad (\PP^2,\conic),\quad (F_n,C_n+f),\,n\geq 0\,.
\end{equation}

\subsection{Quiver}
For each pair as above, there is a certain type of quivers associated to it by \cite{Bou21}. We give a brief review of the construction in Section \ref{subsec:GW/Kr}. The results can be summarized as follows.
For $(\PP^2,\lne)$, the corresponding quiver is 
\begin{center}
\includegraphics[width=\textwidth]{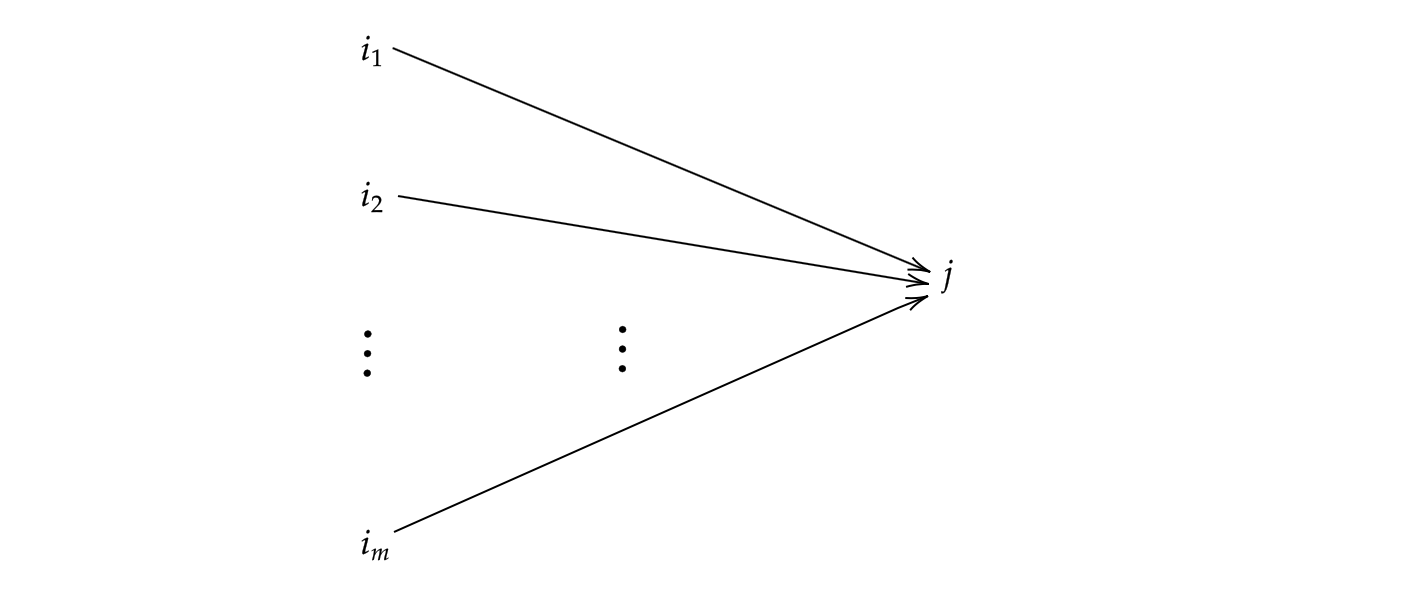}
\end{center}
For $(\PP^2,\conic)$, the corresponding quiver is 
\begin{center}
\includegraphics[width=\textwidth]{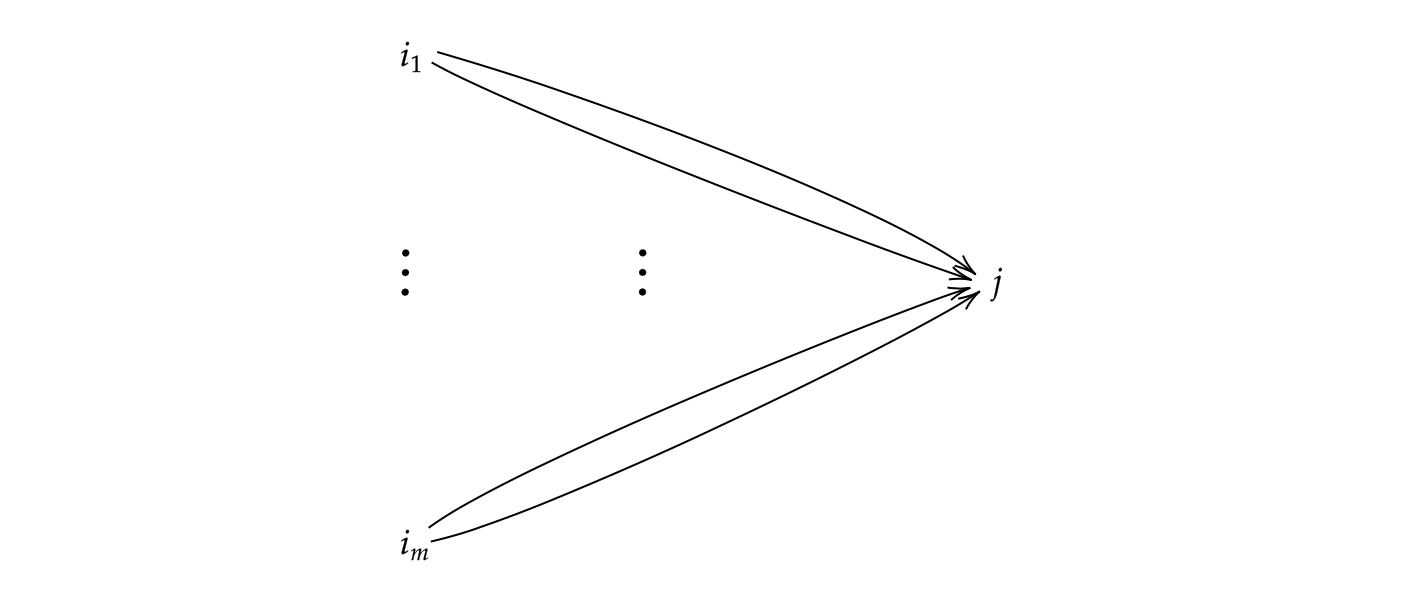}
\end{center}
As for $(F_n,C_n+f)$, the corresponding quiver is 
\begin{center}
\includegraphics[width=\textwidth]{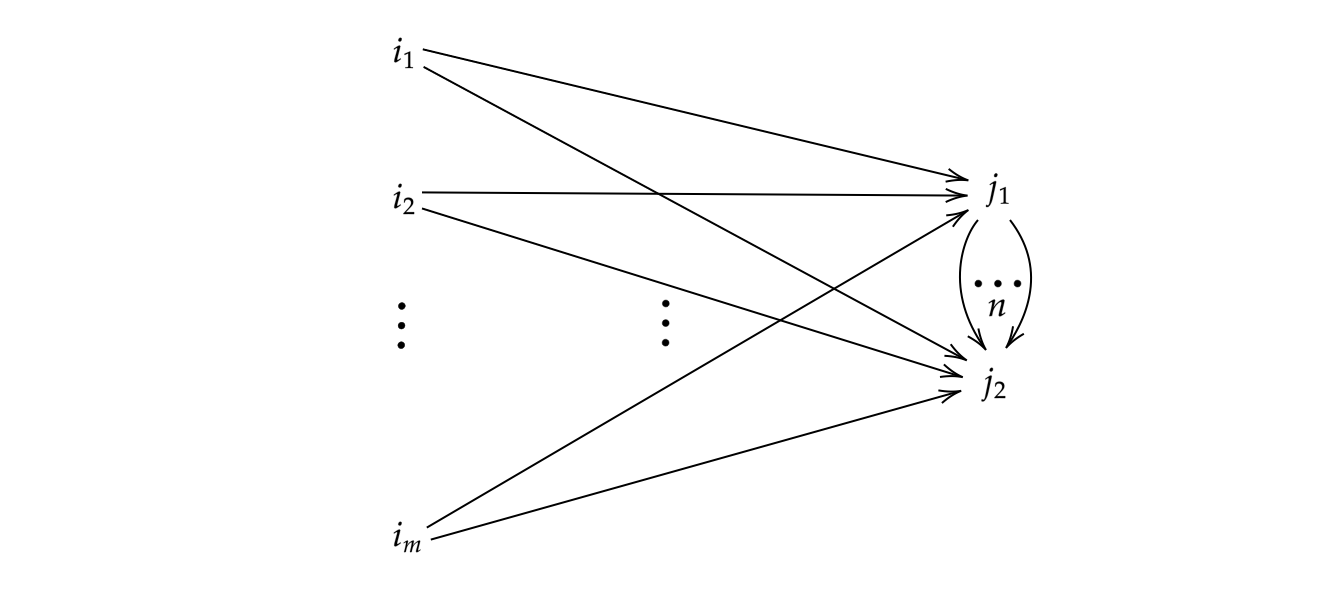}
\end{center}
Note that for all the three types of quivers, the number $m$ of vertices on the left-hand side is not fixed. In the third type of quivers, the number $n$ of arrows between vertices $j_1$ and $j_2$ is the same as the subscript of $F_n$. To get a quiver moduli, we need to specify the dimension vector. For those vertices on the left-hand side, we always put dimension one. For those vertices indexed by $j,\, j_1,\,j_2$, the dimensions can be arbitrary non-negative integers, we denote them as $d,\,d_1,\,d_2$ respectively.

We use $M_{m,d}^{L}$, $M_{m,d}^{C}$ and $M_{m,d_1,d_2}^{F_n}$ to denote the moduli spaces of $\theta$-semistable quiver representations associated to the first, second and third type of quivers with stability condition $\theta$ always given by \[\theta(\cdot)=\{\bard,\cdot\}\]
where $\{\cdot,\cdot\}$ is the antisymmetrized Euler form and $\bard$ is the dimension vector associated to the corresponding quiver (see Section \ref{subsec:quiverintr} for more details). 

Given a projective moduli space $Y$ of semistable quiver representations, the corresponding refined Donaldson-Thomas invariant $\Omega_Y(q)$ is defined as follows. If the stable locus of $Y$ is not empty, then
\begin{equation}\label{eqn:dt}
\Omega_Y(q)=(-q^{1/2})^{-\dim_{\C}Y} \sum_{i=0}^{2 \dim_\C Y} \dim\text{IH}^i(Y,\Q)(-q^{1/2})^i\end{equation}
i.e., $\Omega_Y(q)$ is the shifted Poincar\'e polynomial of the intersection cohomology of $Y$. Otherwise, if the stable locus of $Y$ is empty, $\Omega_{Y}(q)=0$. The equivalence with the definition of quiver Donaldson--Thomas invariants 
based on the motivic Hall algebra \cite{JS, KS, R10}
is shown in \cite{MR}.

\begin{thm}[\cite{Bou21}]\label{thm:GW/Kr}
For the following three types of pairs $(X,D)$:
\begin{equation*}
(\PP^2,\lne),\quad (\PP^2,\conic),\quad (F_n,C_n+f),\,n\geq 0 \,,
\end{equation*}
we have
$$\Omega_{M_{\beta}^{X/D}}(q)=F_{\beta}^{X/D}\frac{(-1)^{D\cdot\beta+1}}{(2\sin(h/2))^{\Tg\cdot\beta-1}}$$
if $\Tg\cdot\beta>0$, where $q=e^{\sqrt{-1}h}$,  and $M_{\beta}^{X/D}$ is determined by $(X,D)$ and $\beta$ as follows:
\begin{enumerate}
\item if $(X,D)=(\PP^2,\lne)$ and $\beta=d[l]$, then $M_{\beta}^{X/D}=M^{L}_{2d,d}$ where $[l]$ is the line class;
\item if $(X,D)=(\PP^2,\conic)$ and $\beta=d[l]$, then $M_{\beta}^{X/D}=M^{C}_{d,d}$;
\item if $(X,D)=(F_n,C_n+f)$ and $\beta=d_1C_{-n}+d_2f$, then $M_{\beta}^{X/D}=M^{F_n}_{m,d_1,d_2}$ where $m=\Tg\cdot\beta=(1-n)d_1+d_2$.
\end{enumerate}
\end{thm}

Combining Theorems \ref{thm:loc/rel} and \ref{thm:GW/Kr}, we can translate the recursion from local side to quiver side. 
We explain in Section \ref{sec:proof} how to derive the recursion on the quiver side from the following geometric properties of the quiver moduli spaces.

\subsubsection{Duality}
Using reflection functors, we obtain the following isomorphisms:
\begin{thm}\label{thm:dual}
\begin{enumerate}
\item $M_{m,d}^L\simeq M_{m,m-d}^L$ if $m$, $d$ are coprime and $m-d>0$;
\item $M_{m,d}^C\simeq M_{m,2m-d}^C$ if $m$, $d$ are coprime and $2m-d>0$;
\item $M_{m,d_1,d_2}^{F_0}\simeq M_{m,m-d_1,m-d_2}^{F_0}$ if $m>\max\{d_1,d_2\}$ and $m$, $d_1+d_2$ are coprime;
\item $M_{(1-n)d_1+d_2+1,d_1,d_2}^{F_n}\simeq M_{(1-n)d_1+d_2+1,d_1+1,d_2+n+1}^{F_n}$ if $n>0$ and $(1-n)d_1+d_2\geq 0$.
\end{enumerate}
\end{thm}

\subsubsection{Other geometric properties}
We also need other geometric properties which can be summarized as follows. Let $M_{m,d}^{L,\text{fr}}$, $M_{m,d}^{C,\text{fr}}$, $M_{m,d_1,d_2}^{F_n,\text{fr}}$ be the framed quiver moduli associated to $M_{m,d}^{L}$, $M_{m,d}^{C}$, $M_{m,d_1,d_2}^{F_n}$ respectively (see Section 3.4 for more details). 
\begin{thm}\label{thm:f/uf}We have the following isomorphisms\
\begin{enumerate}
\item $M_{m,d}^{L,fr}\simeq M_{m+1,d}^L$ if $m$ divides $d$;
\item $M_{m,d}^{C,fr}\simeq M_{m+1,d}^{C}$ if $m$ divides $d$;
\item $M_{m,d_1,d_2}^{F_0,fr}\simeq M_{m+1,d_1,d_2}^{F_0}$ if $m$ divides $d_1+d_2$;
\item $M_{(1-n)d_1+d_2,d_1,d_2}^{F_n,fr}\simeq M_{(1-n)d_1+d_2+1,d_1,d_2}^{F_n}$ if $n>0$ and $(1-n)d_1+d_2\geq 0$.
\end{enumerate}
\end{thm}

\begin{thm} \label{thm:smres}
Assume that $m>0$. There exist small resolutions:\
\begin{enumerate}
\item $M_{m-1,d}^{L,fr}\rightarrow M_{m,d}^{L}$ if $m$ divides $d$ and $m>d$;
\item $M_{m-1,d}^{C,fr}\rightarrow M_{m,d}^{C}$ if $m$ divides $d$;
\item $M_{m-1,d_1,d_2}^{F_0,fr}\rightarrow M_{m,d_1,d_2}^{F_0}$ if $m$ divides $d_1+d_2$ and $d_1,d_2>0$;
\item $M_{(1-n)d_1+d_2-1,d_1,d_2}^{F_n,fr}\rightarrow M_{(1-n)d_1+d_2,d_1,d_2}^{F_n}$ if $n>0$, $d_1>0$ and $d_2>\max\{d_1-1,(n-1)d_1\}$.
\end{enumerate}
\end{thm}

\begin{rmk}
All the geometric properties of case (1) in Theorems \ref{thm:dual}, \ref{thm:f/uf} and \ref{thm:smres} were first proven by Reineke and Weist in \cite{RW21}. We show in Section \ref{sec:GW/Kr} that cases (2), (3), (4) in Theorems \ref{thm:dual}, \ref{thm:f/uf} and \ref{thm:smres} can be deduced via a slight generalization of their
argument in \cite{RW21}.
\end{rmk}

\subsection{BPS invariants}
The BPS invariants $n_{g,\beta}^{\mathcal{O}_X(-D)}$ for $\mathcal{O}_X(-D)$ are defined via 
\begin{equation}\label{eqn:defbps}
F_{\beta}^{\mathcal{O}_X(-D)}=\sum_g n_{g,\beta}^{\mathcal{O}_X(-D)}(2\sin(h/2))^{2g-2+\Tg\cdot\beta}
\end{equation}
if $\Tg\cdot\beta>0$. Combining the deformation equivalence and Theorems \ref{thm:loc/rel}, \ref{thm:GW/Kr} and \ref{thm:smres}, we can deduce that 
\begin{thm}\label{thm:bps}
Under the same assumptions as in Theorem \ref{thm:main} for the pair $(X,D)$, we have
\begin{equation}\label{eqn:bps}
\sum_{g} n_{g,\beta}^{\mathcal{O}_X(-D)}(2\sin(h/2))^{2g}=(-1)^{D\cdot\beta-1}\Omega_{M_{\beta}^{\cO_X(-D)}}(q)
\end{equation}
if $\Tg\cdot\beta>0$, where $q=e^{\sqrt{-1}h}$, and $M^{\cO_X(-D)}_{\beta}$ is determined by $(X,D)$ and $\beta$ as follows:
\begin{enumerate}
\item if $(X,D)=(\PP^2,\lne)$ and $\beta=d[l]$, then $M_{\beta}^{\cO_X(-D)}=M^{L}_{2d-1,d}$;
\item if $(X,D)=(\PP^2,\conic)$ and $\beta=d[l]$, then $M_{\beta}^{\cO_X(-D)}=M^{C}_{d-1,d}$;
\item if $(X,D)=(F_n,C_n+(s+1)f)$ and $\beta=d_1C_{-n}+d_2f$, then $M_{\beta}^{\cO_X(-D)}=M^{F_{n+2s}}_{m-1,d_1,d_2+sd_1}$ where $m=\Tg\cdot\beta=(1-n-s)d_1+d_2$.
\end{enumerate}
\end{thm}
By computing the dimension of $M_{\beta}^{\cO_X(-D)}$, we are able to determine the BPS Castelnuovo number\footnote{Our definition of the BPS Castelnuovo number is different from that of Doan and Walpuski in \cite{DW19} where they defined the BPS Castelnuovo number to be $\inf\{g\,|\,n_{g,\beta}^{\mathcal{O}_X(-D)}= 0\}$.} defined as
\[g_{\beta}^{\mathcal{O}_X(-D)}\coloneqq \sup\{g\,|\,n_{g,\beta}^{\mathcal{O}_X(-D)}\neq 0\}.\]
\begin{cor}\label{cor:bpsca}
Under the same assumptions as in Theorem \ref{thm:main} for the pair $(X,D)$ and assuming $\Tg\cdot\beta>0$, $M_{\beta}^{\cO_X(-D)}\neq \emptyset$, we have
\begin{enumerate}
\item $g_{\beta}^{\mathcal{O}_X(-D)}=\frac{(K_X+\beta)\cdot\beta}{2}+1$;
\item $n_{g,\beta}^{\mathcal{O}_X(-D)}=(-1)^{g+D\cdot\beta-1}$, if $g=\frac{(K_X+\beta)\cdot\beta}{2}+1\geq 0$.
\end{enumerate}
\end{cor}
Note that case (1) of Corollary \ref{cor:bpsca} matches with the genus-degree formula, and case (2) of Corollary \ref{cor:bpsca} actually follows from the geometric fact that the moduli space $M_{\beta}^{\cO_X(-D)}$ is connected.

\subsection{Numerical results}
Theorem \ref{thm:main}  gives an effective way to compute $F_{\beta}^{\cO_X(-D)}$ once we know those $F_{\beta}^{\cO_X(-D)}$ such that $\Tg\cdot\beta<3$. We show in Section \ref{sec:numdata} that 
for the following five pairs:
\[(\PP^2,\lne),\, (\PP^2,\conic),\,(F_n,C_n+f),\, n=0,1,2\,,\]
those initial $F_{\beta}^{\cO_X(-D)}$ can be explicitly determined from the quiver side. We also give some numerical results for the above five pairs in Section \ref{sec:numdata}. As for the pairs 
\[(F_n,C_n+f),\,n>2,\]
the initial $F_{\beta}^{\cO_X(-D)}$ are related to the quiver DT-invariants of $n$-Kronecker quivers with $n>2$,  for which no explicit closed formula is currently known (see Section \ref{sec:numdata} for more details).
\subsection{Organization of the paper}
This paper is organized as follows. In Section \ref{sec:loc/rel}, we give a brief introduction of local/relative correspondence and prove Theorem \ref{thm:loc/rel}. In Section 
\ref{sec:GW/Kr}, after a brief review of GW/quiver correspondence, we prove Theorems \ref{thm:dual}, \ref{thm:f/uf} and \ref{thm:smres}. In Section \ref{section:bps}, combining local/relative correspondence with GW/quiver correspondence, we prove Theorem \ref{thm:bps} and Corollary \ref{cor:bpsca} on the BPS invariants. In Section \ref{sec:proof}, we prove our main theorem, i.e., Theorem \ref{thm:main} on the all-genus WDVV recursion. In Section \ref{sec:numdata}, we give some numerical results based on the all-genus WDVV recursion. In Appendix \ref{sec:ex}, we give a counterexample which shows that the ample condition in Theorem \ref{thm:main} can not be relaxed to be nef. In Appendix \ref{sec:recg0}, we give another proof of Corollary \ref{cor:g0rec} purely on the Gromov-Witten side. In Appendix \ref{app:virasoro}, we give a comparison of our recursion with a recursion derived from the Virasoro constraints.  
\subsection{Acknowledgment}
We would like to thank Rahul Pandharipande, Honglu Fan, Shuai Guo, Hyenho Lho, and Michel van Garrel for helpful discussions. This project was supported by the National Key R\&D Program of China (No. 2022YFA1006200).

\section{Local/relative correspondence}\label{sec:loc/rel}
Let $D$ be a rational smooth ample divisor in a smooth projective surface $X$ over $\C$. In this section, as a first step towards the recursion formula \eqref{eqn:recurfm}, we prove Theorem \ref{thm:loc/rel} which relates generating series of local Gromov--Witten invariants of $\cO_X(-D)$ to generating series of relative Gromov--Witten invariants of the pair $(X,D)$. 

\subsection{Local/relative correspondence} We will prove Theorem \ref{thm:loc/rel} by applying the higher genus local/relative correspondence established in \cite{BFGW} which expresses the virtual cycle of $\bM_{g}(\mathcal{O}_X(-D),\beta)$ as 
\[\frac{(-1)^{\beta \cdot D-1}}{\beta \cdot D}F_*\left((-1)^g\lambda_g \cap [\bM_{g}(X/D,\beta)]^{\vir}\right)+ \sum_{\sG\in G_{g,\beta}}\frac{1}{|\Aut(\sG)|}(\tau_{\sG})_*\left(C_{\sG}\cap [\bM_{\sG}]^{\vir}\right)\,,\]
where the various terms are described as follows.
The map $F: \bM_{g}(X/D,\beta)\rightarrow \bM_g(X,\beta)$ is the obtained by forgetting the unique relative marking and stabilizing. Moreover, $G_{g,\beta}$ is a set of star type graphs $\sG=(V,E,g,b)$ of the following form: the set $V$ of vertices of $\sG$ admits a decomposition $\{v\}\coprod V_1$ such that 
for every $v_i \in V_1$, there is a unique edge between $v$ and $v_i$, and all edges of $\sG$ are of this form. In addition, we have maps
$\g:V\rightarrow \Z_{\geq 0}$ and $b:V\rightarrow H_2(D,\Z)\cup H_2(X,\Z)$ assigning a genus and a curve class to every vertex, such that $b(v) \in H_2(D,\Z)$, $b(v_i) \in H_2(X,\Z)$ for every $v_i \in V_1$, $\sum_{v\in V}\g(v)=g$ and $\sum_{v_i\in V_1}b(v_i)+\iota_*(b(v))=\beta$ where $\iota:D\hookrightarrow X$ is the inclusion of $D$ in $X$.

The moduli space $\bM_{\sG}$ decorated by a graph $\sG \in G_{g,\beta}$ is
\[\bM_{\sG}=\left(\prod_{v_i}\bM_{\g(v_i)}(X/D,b(v_i))\right)\times_{D^{|E|}}\bM_{\g(v),|E|}(D,b(v))\,, \] 
where $|E|$ is the number of edges of $\sG$. 
The virtual cycle of $\bM_{\sG}$ is defined using the Gysin map associated to the diagonal map $\Delta:D^{|E|}\rightarrow D^{|E|}\times D^{|E|}$ as:
\[[\bM_{\sG}]^{\vir} := \Delta^!\left(\prod_{v_i\in V_1}\left[\bM_{\g(v_i)}(X/D,b(v_i))\right]^{\vir}\times \left[\bM_{\g(v),|E|}(D,b(v))\right]^{\vir}\right).\]
Finally, $C_{\sG}$ is a certain class in the Chow ring of $\bM_{\sG}$ containing a factor of $\prod_{v_i \in V_1} \lambda_{g_{v_i}}$ (see \cite[Section 2.2]{BFGW} for more details), and $\tau_{\sG}$ is the natural gluing map from $\bM_{\sG}$ to $\bM_g(X,\beta)$.

In our case, the definition of local invariants $N_{g,\beta}^{\cO_X(-D)}$ also includes the insertions of point classes. So we need to generalize the above expression for $[\bM_{g}(\mathcal{O}_X(-D),\beta)]^{\vir}$ to the case of 
\begin{equation}\label{eqn:markins}
\prod_{i=1}^m ev_i^*([pt])\cap [\bM_{g,m}(\mathcal{O}_X(-D),\beta)]^{\vir}.
\end{equation}
Following the same argument as in \cite{BFGW}, one obtains that \eqref{eqn:markins} can be expressed as 
\[\prod_{i=1}^m ev_i^*([pt])\cap\left(\frac{(-1)^{\beta \cdot D-1}}{\beta \cdot D}F_*\left((-1)^g\lambda_{g} \cap [\bM_{g,m}(X/D,\beta)]^{\vir}\right)+ \sum_{\sG\in G_{g,\beta,m}}\frac{1}{|\Aut(\sG)|}(\tau_{\sG})_*\left(C_{\sG}\cap [\bM_{\sG}]^{\vir}\right)\right).\]
Here each graph $\sG=(V,E,g,b,a)\in G_{g,\beta,m}$ contains an addition data $a$ which gives an assignment of the $m$ markings to the vertices in $V_1$. Let $|a^{-1}(v_i)|$ denote the number of additional markings assigned to the vertex $v_i\in V_1$. Then $\bM_{\sG}$ becomes
\[\left(\prod_{v_i}\bM_{\g(v_i),|a^{-1}(v_i)|}(X/D,b(v_i))\right)\times_{D^{|E|}}\bM_{\g(v),|E|}(D,b(v))\] 
and the definition of $C_{\sG}$ is the same as in the $m=0$ case. 

\subsection{Proof of Theorem \ref{thm:loc/rel}} Now let $m=\Tg\cdot\beta$ and apply the degree map to \eqref{eqn:markins}, we get $N_{g,\beta}^{\cO_X(-D)}$. If we apply the degree map to
\[\prod_{i=1}^m ev_i^*([pt])\cap\left(\frac{(-1)^{\beta \cdot D-1}}{\beta \cdot D}F_*\left((-1)^g\lambda_{g} \cap [\bM_{g,m}(X/D,\beta)]^{\vir}\right)\right)\]
then we get $\frac{(-1)^{D\cdot\beta-1}}{D\cdot\beta}N_{g,\beta}^{X/D}$. We are left to study
\begin{equation}\label{eqn:corrtms}
\prod_{i=1}^m ev_i^*([pt])\cap\left( \sum_{\sG\in G_{g,\beta,m}}\frac{1}{|\Aut(\sG)|}(\tau_{\sG})_*\left(C_{\sG}\cap [\bM_{\sG}]^{\vir}\right)\right).
\end{equation}
Since all the markings are assigned to the vertices in $V_1$ and $C_{\sG}$ always contains a factor $\prod_{v_i\in V_1}\lambda_{\g(v_i)}$, the cohomology class assigned to the factor $\prod_{v_i\in V_1}\bM_{\g(v_i),|a^{-1}(v_i)|}(X/D,b(v_i))$ in $\bM_{\sG}$ has degree at least
\[\deg\left( \prod_{i=1}^m ev_i^*([pt])\prod_{v_i\in V_1}\lambda_{\g(v_i)}\right)=2m+\sum_{v_i\in V_1}\g(v_i)\,.\]
Given the virtual dimension
\[\dim \left(\prod_{v_i\in V_1}\left[\bM_{\g(v_i),|a^{-1}(v_i)|}(X/D,b(v_i))\right]^{\vir}\right) =m+\sum_{v_i\in V_1}\left(\Tg\cdot b(v_i)+\g(v_i)\right)\,,\]
in order to get a nonzero contribution to \eqref{eqn:corrtms}, we need 
\[m+\sum_{v_i\in V_1}\left(\Tg\cdot b(v_i)+\g(v_i)\right)\geq 2m+\sum_{v_i\in V_1}\g(v_i)\,,\]
which is equivalent to
\[\sum_{v_i\in V_1}\Tg\cdot b(v_i)\geq m=\Tg\cdot\beta=\Tg\cdot\iota_*(b(v))+\sum_{v_i\in V_1}\Tg\cdot b(v_i),\]
i.e., $\Tg\cdot\iota_*(b(v))\leq 0$.
As $D$ is a curve, $\iota_*(b(v))$ is a multiple of the fundamental class of $D$, and so
$\Tg\cdot\iota_*(b(v))=n\Tg\cdot D$ where $n\in \Z_{\geq 0}$ is the degree of $b(v)\in H_2(D,\Z)$. The condition that $D$ is rational further implies that $\Tg\cdot D=2>0$ by the adjunction formula. So we must have $n=0$. In the this case, it is not possible to add further insertions on the factor $\prod_{v_i\in V_1}\bM_{\g(v_i),|a^{-1}(v_i)|}(X/D,b(v_i))$. So if we apply the degree map to \eqref{eqn:corrtms}, the K\"unneth decomposition of the diagonal class $[\Delta]\in H^2(D^{|E|}\times D^{|E|}, \Z)$ contributes $|E|$ insertions of the point class of $D$ to the factor $\bM_{\g(v),|E|}(D,b(v))$ in $\bM_{\sG}$. Since $b(v)=0$,
it is not possible to have more than one point class insertions to 
$\bM_{\g(v),|E|}(D,0)$ for a nonzero invariant. This implies that $|E|=1$.

From the above discussion, we conclude that if we apply the degree map to \eqref{eqn:corrtms}, then only those $\sG\in G_{g,\beta,m}$ such that 
\[\bM_{\sG}=\bM_{g_1,m}(X/D,\beta)\times_{D}\bM_{g_2,1}(D,0)\]
give nonzero contributions where $g_1+g_2=g$, $g_2>0$. By the description of the class   
$C_{\sG}$
in \cite[Section 2.2]{BFGW}, the corresponding contribution is 
\[N_{g_1,\beta}^{X/D}(-1)^{g_2-1+D\cdot\beta}(D\cdot\beta)^{2g_2-1}\int_{[\bM_{g_2,1}(D,0)]^{\vir}}\psi_1^{2g_2-2}ev_1^*(w)\,.\]
Here $w$ is the point class of $D$ and $\psi_1$ is the psi-class associated to the unique marking. The degree zero Gromov--Witten invariant is determined by a Hodge integral on the moduli space of curves:
\[\int_{[\bM_{g_2,1}(D,0)]^{\vir}}\psi_1^{2g_2-2}ev_1^*(w)=(-1)^{g_2}\int_{[\bM_{g_2,1}]}\psi_1^{2g_2-2}\lambda_{g_2},,\]
and so we conclude that
\begin{equation}\label{eqn:loc/rel_num}
N_{g,\beta}^{\cO_X(-D)}=\sum_{g_1+g_2=g}N_{g_1,\beta}^{X/D}(-1)^{D\cdot\beta-1}(D\cdot\beta)^{2g_2-1}\int_{[\bM_{g_2,1}]}\psi_1^{2g_2-2}\lambda_{g_2}.
\end{equation}
Here we simply set $\int_{[\bM_{g_2,1}]}\psi_1^{2g_2-2}\lambda_{g_2}=1$ when $g_2=0$ so as to include the term $\frac{(-1)^{D\cdot\beta-1}}{D\cdot\beta}N_{g,\beta}^{X/D}$ in the right-hand side of \eqref{eqn:loc/rel_num}. The Hodge integrals appearing in \eqref{eqn:loc/rel_num} have been computed by Faber and Pandharipande in \cite{FP}: Theorem \ref{thm:loc/rel}:
\[F_{\beta}^{\cO_X(-D)}=F_{\beta}^{X/D}\frac{(-1)^{D\cdot\beta-1}}{2\sin(\frac{(D\cdot\beta)h}{2})}\]
follows from \eqref{eqn:loc/rel_num} and \cite[Theorem 2]{FP}.

\section{GW/quiver correspondence}\label{sec:GW/Kr}

In this section, after a review of quiver representations and 
of the GW/quiver correspondence of \cite{Bou21}, we prove Theorems \ref{thm:dual}, \ref{thm:f/uf}, \ref{thm:smres} on geometric properties of moduli spaces of quiver representations. 

\subsection{Quiver moduli}\label{subsec:quiverintr}
We give a brief introduction to the moduli of representations of quivers in this section. The main reference is \cite{R08}. We always work over $\C$.

A quiver $Q$ consists of a finite set of vertices $Q_0$ and a finite set of arrows $Q_1=\{\alpha: i\rightarrow j|i,j\in Q_0\}$. Let
\[\Z Q_0\coloneqq \bigoplus_{i\in Q_0} \Z e_i\]
where $\{e_i\}_{i\in Q_0}$ forms a natural basis. The Euler form on $\Z Q_0$ is defined by
\[\langle\underline{a},\underline{b}\rangle\coloneqq \sum_{i\in Q_0}a_ib_i-\sum_{\alpha:i\rightarrow j\in Q_1}a_ib_j\]
for $\underline{a}=(a_i)_{i\in Q_0},\,\underline{b}=(b_i)_{i\in Q_0}\in \Z Q_0$. We further use $(\underline{a},\underline{b})\coloneqq\langle\underline{a},\underline{b}\rangle+\langle\underline{b},\underline{a}\rangle$ to denote symmetrized Euler form and use $\{\underline{a},\underline{b}\}\coloneqq\langle\underline{a},\underline{b}\rangle-\langle\underline{b},\underline{a}\rangle$ to denote the antisymmetrized Euler form.

A representation of a quiver $Q$ consists of a tuple of vector spaces $(V_i)_{i\in Q_0}$ indexed by the vertices, plus a tuple of linear morphisms $(V_{\alpha}:V_i\rightarrow V_j)_{\alpha:i\rightarrow j}$ indexed by the arrows. A morphism between two representations $V,W$ consists of a tuple of linear morphisms $\{f_i:V_i\rightarrow W_i\}_{i\in Q}$ such that $f_j\circ V_\alpha=W_{\alpha}\circ f_i$ holds for arbitrary arrow $\alpha:i\rightarrow j$.
The complex representations of a fixed quiver $Q$ form a category via componentwise composition. We denote it as $\Rep_{\C}Q$. Actually, $\Rep_{\C}Q$ is a $\C$-linear abelian category with finite length.

Let $\underline{d}=(\dim V_i)_{i\in Q_0}$ be a dimension vector. The vector space
\[R_{\bard}(Q)\coloneqq \bigoplus_{\alpha:i\rightarrow j}\Hom(V_i,V_j)\]
parametrizes all the representations of $Q$ with fixed dimensions. The group
\[G_{\bard}(Q)\coloneqq\bigoplus_{i\in Q_0} GL(V_i) \]
naturally acts on $R_{\bard}$ as follows:
\begin{eqnarray*}
G_{\bard}(Q)\times R_{\bard}(Q) & \longrightarrow & R_{\bard}(Q)\\
(g_i)_{i\in Q_0}\circ (V_{\alpha})_{\alpha:i\rightarrow j} & \longmapsto & (g_j\circ V_{\alpha}\circ g_i^{-1})_{\alpha:i\rightarrow j}
\end{eqnarray*}
We want to construct a moduli which parametrizes the isomorphism classes of quiver representations. The quotient of $R_{\bard}$ by $G_{\bard}$ might not yield an interesting moduli space. For example, if $Q$ is acyclic, that is, if there is no directed cycles in $Q$, then \[R_{\bard}//G_{\bard}:=\Spec(\C[R_{\bard}]^{G_{\bard}})\]
is just a point. To have a good moduli of quiver representations, we need to introduce a stability condition. A \emph{stability condition} 
\[\theta=\sum_{i\in Q_0}\theta_i e_i^*\]
is an element in the dual space $(\Z Q_0)^*$, where  $\{e_i^*\}_{i\in Q_0}$ is the dual basis, that is, $\theta(\bard)=\sum_{i\in Q_0}\theta_id_i$.

A stability condition $\theta$ induces a \emph{slope function} $\mu:\N Q_0\setminus \{0\}\rightarrow \Q$ given by \
\[\mu(\bard)=\frac{\theta(\bard)}{|\bard|}\] 
where $|\bard|=\sum_{i\in Q_0} d_i$. A representation $V$ is called \emph{$\theta$-semistable} (resp. \emph{$\theta$-stable})
if $\mu(U)\leq \mu(V)$ (resp. $\mu(U)< \mu(V)$) for all the nonzero proper subrepresentations $U$ of $V$.  We use $R_{\bard}^{\theta - sst}$ (resp. $R_{\bard}^{\theta - st}$)
to denote the set of $\theta$-semistable (resp. $\theta$-stable) representation in $R_{\bard}$. The quotient $M_{\bard}^{\theta-sst}=R_{\bard}^{\theta- sst}//G_{\bard}$ parametrizes the closed orbits.\footnote{The closed orbits are in one-to-one correspondence with the isomorphism classes of $\theta$-polystable representations.} The moduli space $M_{\bard}^{\theta-sst}$ contains a Zariski open subset $M_{\bard}^{\theta-st}=R_{\bard}^{\theta- st}/G_{\bard}$ which parametrizes $\theta$-stable representations in $R_{\bard}$. Both $M_{\bard}^{\theta-sst}$ and $M_{\bard}^{\theta-st}$ are irreducible with dimension $1-<\bard,\bard>$ if $M_{\bard}^{\theta-st}$ is not empty.

Without further mention, the stability condition in this paper is always assumed to be $\theta(\cdot)=\{\bard,\cdot\}$ for a given dimension vector $\bard$. Using the terminology of \cite{APflow, MPattr, ABflow}, $\theta=\{\bard,\cdot\}$ is the 
\emph{anti-attractor} stability condition.







\subsection{GW/quiver correspondence}\label{subsec:GW/Kr}
In this subsection, we explain how Theorem 
\ref{thm:GW/Kr} follows from the GW/quiver correspondence of \cite{Bou21}.

The set-up of \cite{Bou21} is a smooth projective surface $Y$ over $\C$, with two smooth non-empty divisors $D_1$ and $D_2$ intersecting transversally, and such that the union $D_1\cup D_2$ is anticanonical.
We consider the following maximal contact relative Gromov--Witten invariants of the pair $(Y,D_1)$: 
\[N_{g,\beta}^{Y/D_1}\coloneqq \int_{[\bM_{g,m}(X/D,\beta)]^{\vir}}\prod_{i=1}^m ev_i^*([pt])(-1)^g \lambda_g\]
where $\beta$ is a curve class on $Y$ such that $\beta \cdot D_1>0$ and $\beta \cdot D_2>0$, $m=\beta \cdot D_2$, and $\lambda_g$ is the top Chern class of the Hodge bundle. The main result of \cite{Bou21} is the construction of a quiver $Q_\beta^{Y/D_1}$ and of a dimension vector $d(\beta)$ such that, if 
$Q_\beta^{Y/D_1}$ is acyclic, we have by 
\cite[Theorem 1.2]{Bou21}
\begin{equation} \label{eq:gw_dt}
\Omega_{M_\beta}(q)=(-1)^{\beta \cdot D_1+1} \frac{1}{(2 \sin(h/2))^{\beta \cdot D_2 -1}} \sum_{g \geq 0} N_{g,\beta}^{Y/D_1} \hbar^{2g-1+\beta \cdot D_2} \,,\end{equation}
where $M_\beta$ is the moduli space of $\theta$-semistable representations of 
$Q_\beta^{Y/D_1}$ of dimension $d(\beta)$, where $\theta=\{ d(\beta),-\}$, and $q=e^{\sqrt{-1}h}$.

One can take $(Y,D_1)=(X,D)$ for the first two pairs in Theorem \ref{thm:GW/Kr}: see \cite[6.1]{Bou21} for 
$(\PP^2,\lne)$, and \cite[6.2]{Bou21} for $(\PP^2,\conic)$. For each case, the quiver $Q_\beta^{Y/D_1}$ and the dimension vector $d(\beta)$ are given as in the statement of Theorem \ref{thm:GW/Kr}. Using that $\beta\cdot D_2=\beta \cdot (-K_X-D)=\Tg\cdot\beta$, one deduces that Theorem \ref{thm:GW/Kr} for these pairs follows from \eqref{eq:gw_dt}. 

For $(F_n,C_n+f)$, we use that this pair is deformation equivalent to the pair $(F_{n-2}, C_{n-2}+2f)$, the class $\beta=d_1 C_{-n}+d_2 f$ on $F_n$ deforming to the class $\beta'
=d_1 C_{-(n-2)}+d_2 f$ on $F_{n-2}$. The pair $(Y,D_1)=(F_{n-2}, C_{n-2}+2f)$ is considered in \cite[6.10]{Bou21}. The quiver $Q_{\beta'}^{Y/D_1}$ and the dimension vector $d(\beta')$ given in \cite[6.10]{Bou21} are related to the quiver and dimension vector in Theorem \ref{thm:GW/Kr} by mutation of the bottom right vertex. In general, one defines in \cite{Bou21} a quiver $Q_{\beta'}^{Y/D_1}$ for each choice of toric model of $(Y,D_1 \cup D_2)$ and mutations of quivers correspond to changes of toric models. Therefore, the quiver and dimension vector in Theorem \ref{thm:GW/Kr}
can be obtained directly from the construction of \cite{Bou21} for a different choice of toric model.
One concludes that Theorem \ref{thm:GW/Kr} for $(F_n,C_n+f)$ also follows from \eqref{eq:gw_dt}. 

\begin{rmk}
There are unfortunately two small errors in the statement of \cite[Theorem 1.2]{Bou21} as written in the published version of \cite{Bou21}. First, the pre-factor containing $2 \sin(h/2)$ is not the right one. Second, quiver DT invariant are taken in with respect to a 
``maximally non-trivial stability condition" rather than with respect to the anti-attractor stability condition. Both errors are corrected in the most recent arXiv version of \cite{Bou21}.
\end{rmk}

\subsection{Duality}\label{subsec:dual}
In this subsection, we prove Theorem \ref{thm:dual}. First, let us give a brief introduction of reflection functors following \cites{BGP,W13}.

For a quiver $Q=(Q_0,Q_1)$, we call a vertex $i\in Q_0$ a \emph{sink} (resp. a \emph{source}) if there are no arrows starting from $i$ (resp. ending on $i$). Let $i$ be a vertex. Then we use $\sigma_iQ$ to denote the quiver by reversing all arrows staring or ending at $i$. We will define a pair of reflection functors 
\[R_i^{+},\, R_i^{-}:\Rep_{\C}Q\longrightarrow \Rep_{\C}
\sigma_i Q\]
corresponding to the sink and source cases respectively. Let $V$, $W$ be two representations of $\Rep_{\C}Q$, and $f:V\rightarrow W$ be a morphism between them.

(I) Let us start from the case when $i$ is a sink. The reflection functor $R_i^{+}$ can be constructed as follows. We set $R_i^{+}V=X$ such that $X_j=V_j$ for $j\neq i$ and $X_i$ is given by the kernel of
\[(V_{\alpha}):\, \bigoplus_{\alpha\in Q_1\atop t(\alpha)=i}V_{s(\alpha)}\longrightarrow V_i\]
where $s,t:\,Q_1\rightarrow Q_0$ are two natural maps such that for each arrow $\alpha$ it starts at $s(\alpha)$ and terminates at $t(\alpha)$. For an arrow $\alpha$ such that $t(\alpha)\neq i$, we set the linear morphism $X_{\alpha}=V_{\alpha}$. Note that if $t(\alpha)=i$, then we need to reverse the direction of arrow $\alpha$ to get 
the arrow $\sigma_i\alpha$ in $\sigma_iQ$. The corresponding linear morphism $X_{\sigma_i\alpha}:X_i\rightarrow V_{s(\alpha)}=X_{t(\sigma_i\alpha)}$ is given by the composition of the natural inclusion and the projection:
$$X_i\hookrightarrow \bigoplus_{\alpha\in Q_1\atop t(\alpha)=i}V_{s(\alpha)}\longrightarrow V_{s(\alpha)}.$$
Similarly, we can construct $R_i^+W=Y$. For the morphism $R_i^+f=g:X\rightarrow Y$, let $g_j=f_j$ if $j\neq i$ and $g_i: X_i\rightarrow Y_i$ is defined to be the restriction of the linear morphism
\[(f_{s(\alpha)}):\, \bigoplus_{\alpha\in Q_1\atop t(\alpha)=i}V_{s(\alpha)}\longrightarrow \bigoplus_{\alpha\in Q_1\atop t(\alpha)=i}W_{s(\alpha)}.\]

(II) Next we consider the case when $i$ is a source. The reflection functor $R_i^-$ can be constructed as follows. Let $R_i^-V=X'$. Still we have $X'_j=V_j$ if $j\neq i$. But $X'_i$
is now given by the cokernel of
\[(V_{\alpha}):\,V_i\longrightarrow \bigoplus _{\alpha\in Q_1\atop s(\alpha)=i}V_{t(\alpha)}.\]
For an arrow $\alpha$ such that $s(\alpha)\neq i$, the linear morphism $X'_{\alpha}=V_{\alpha}$. If $s(\alpha)=i$, then $\sigma_i\alpha$ has an opposite direction as to $\alpha$. The corresponding linear morphism $X'_{\sigma_i\alpha}: X'_{s(\sigma_i\alpha)}=V_{t(\alpha)}\rightarrow X'_i$ is given by the composition of the natural inclusion and the quotient map:
\[V_{t(\alpha)}\hookrightarrow\bigoplus _{\alpha\in Q_1\atop s(\alpha)=i}V_{t(\alpha)}\longrightarrow X'_i.\]
Let $R_i^-W=Y'$. For the morphism $R_i^-f=g':X'\rightarrow Y'$, let $g'_j=f_j$ if $j\neq i$ and $g'_i:X'_i\rightarrow Y'_i$ is induced from \[(f_{t(\alpha)}):\, \bigoplus_{\alpha\in Q_1\atop s(\alpha)=i}V_{t(\alpha)}\longrightarrow \bigoplus_{\alpha\in Q_1\atop s(\alpha)=i}W_{t(\alpha)}.\]
The reflection functors satisfy the following properties.
Let $E_i$ be the simple representation such that $(E_i)_i=\C$ and $(E_i)_j=0$ if $j\neq i$.
\begin{thm}[\cite{BGP}]\label{thm:RF}
Let $i$ be a sink (resp. a source) and $V$ be an irreducible representation\footnote{A representation $V$ is called irreducible if $V\neq 0$ and $V=V_1\oplus V_2$ implies that $V_1=0$ or $V_2=0$.}. Then the reflection functor $R_i^+$ has the following properties (if $i$ is a source, then replace $+$ by $-$):
\begin{enumerate}
\item If $V=E_i$, then $R_i^+V=0$;
\item If $V\neq E_i$, then $R_i^+V$ is irreducible and $R_i^-R_i^+ V\simeq V$. Moreover, the dimension vector $\dim R_i^+V=r_i(\dim V)$ where $r_i:\Z Q_0\rightarrow \Z Q_0$ is given by
\[r_i(x)=x-(x,e_i)e_i.\]
Here $e_i$ is the the $i$th coordinate vector and $(\cdot,\cdot)$ is the symmetrized Euler form given before.
\end{enumerate}
\end{thm}

We also need the following dual functor. Let $Q^{op}$ be the opposite quiver obtained from $Q$ by reversing the orientation of arrows. The vector space duality $\Hom_k(-,k)$ induces the natural dual functor:
\[D:\,\Rep_{\C}Q\longrightarrow \Rep_{\C}
 Q^{op}\]
which sends a quiver representation to its dual representation. For a $\theta$-semistable representation, the dual representation is $(-\theta)$-semistable. So it naturally induces the following isomorphism 
\[D:\,M_{\bard}^{\theta-sst}(Q)\simeq M_{\bard}^{(-\theta)-sst}(Q^{op}).\]

Now the isomorphisms appearing in Theorem \ref{thm:dual} are induced from the reflection functors together with the dual functor above.
\begin{proof}[Proof of Theorem \ref{thm:dual}]
Let us start with case (1): $M_{m,d}^L\simeq M_{m,m-d}^L$ if $m$, $d$ are coprime and $m-d>0$; The isomorphism was first established by Reineke and Weist in \cite{RW21} by showing that both $M_{m,d}^L$ and $M_{m,m-d}^L$ are isomorphic to a third geometric quotient space. 

Let us first reinterpret their proof in terms of the reflection functors and the dual functor given above. We will then show  that all the other cases of Theorem \ref{thm:dual} can be deduced via a parallel argument.
The quiver $Q$ in this case is 
\begin{center}
\includegraphics[width=\textwidth]{images/P1.png}
\end{center}
with dimension vector $\underline{d}=\sum_{k=1}^m e_{i_k}+de_j$. $M_{m,d}^L$ is simply the moduli space $M_{\bard}^{\theta-sst}$ of $\theta$-semistable representations associated to the above quiver with fixed dimension vector $\underline{d}$. Note that the stability condition $\theta$ is always assumed to be $\theta(\cdot)=\{\underline{d},\cdot\}$. It is then easy to check that $M_{\bard}^{\theta-sst}=M_{\bard}^{\theta-st}$ since $m,d$ are coprime.

The vertex $j$ is a sink, so we can apply the reflection functor $R_j^{+}$. After further composition with the dual functor $D$, $D\circ R_j^+$ will map a $\theta$-stable representation (which is always irreducible) in $M_{m,d}^L$ to a representation of the above quiver $Q$ with dimension vector $r_j(\underline{d})=\sum_{k=1}^m e_{i_k}+(m-d)e_j$.
Actually, one can further check that $D\circ R_j^+$ induces an isomorphism between $M_{m,d}^L$ and $M_{m,m-d}^L$ (with inverse $R_j^{-}\circ D$) which coincides with the one given by Reineke-Weist in their proof. 
The key point is to show that
for a given representation $V$ of $R_{\underline{d}}(Q)$, $V$ is $\theta$-stable if and only if $D\circ R_j^+(V)$ is a $\theta$-stable representation of $R_{r_j(\underline{d})}(Q)$ which follows from a simple linear algebra fact \cite[Lemma 3.2]{RW21}. Here with a slightly abuse of notation, we still use $\theta$ to denote the corresponding stability condition for $R_{r_j(\underline{d})}(Q)$, i.e., $\theta(\cdot)=\{r_j(\underline{d}),\cdot\}$. This gives a reinterpretation of Reineke-Weist's proof for case (1) using the reflection functors and the dual functor.

For case (2): $M_{m,d}^C\simeq M_{m,2m-d}^C$ if $m$, $d$ are coprime and $2m-d>0$, the corresponding quiver $Q'$ is 
\begin{center}
\includegraphics[width=\textwidth]{images/P2.png}
\end{center}
with dimension vector $\underline{d}=\sum_{k=1}^m e_{i_k}+de_j$. $M_{m,d}^C$ 
is the moduli of $\theta$-stable representations with fixed dimension vector $\underline{d}$. Then similar to case (1), one can show that $D\circ R_j^+$ induces an isomorphism between $M_{m,d}^C\simeq M_{m,2m-d}^C$. The key point is still to show that for a given representation $V$ of $R_{\underline{d}}(Q')$, $V$ is $\theta$-stable if and only if $D\circ R_j^+(V)$ is a $\theta$-stable representation of $R_{r_j(\underline{d})}(Q')$ which still follows from 
the linear algebra fact \cite[Lemma 3.2]{RW21}. Note that here $r_j(\underline{d})=\sum_{k=1}^m e_{i_k}+(2m-d)e_j$

Similar for cases (3) and (4), they correspond to quivers
\begin{center}
\includegraphics[width=\textwidth]{images/HS.png}
\end{center}
with dimension vectors always given by $\underline{d}=\sum_{k=1}^m e_{i_k}+d_1e_{j_1}+d_2e_{j_2}$. 
Both of the isomorphisms are given by the functor $D\circ R_{j_1}^+\circ R_{j_2}^+$. We miss the details here because they are parallel to cases (1) and (2). But we note that 
\[r_{j_1}\circ r_{j_2}(\underline{d})=\sum_{k=1}^m e_{i_k}+((n^2-1)d_2-nd_2+(n+1)m)e_{j_1}+(nd_1+m-d_2)e_{j_2}\]
and $D\circ R_{j_1}^+\circ R_{j_2}^+$ maps representations associated to the above quiver to representations associated to a different quiver which can be derived from the above quiver by simply changing all the arrow directions between vertices $j_1$ and $j_2$.
\end{proof}

\subsection{Other geometric properties}\label{subsec:othergeo}
In this subsection, we prove Theorems \ref{thm:f/uf} and \ref{thm:smres}.

Before we give a proof of Theorem \ref{thm:f/uf}, we need some preparation. Given a quiver $Q=(Q_0,Q_1)$ and a vector of non-negative integers $\underline{n}\in \N Q_0$. According to \cite[Section 2.3]{RW21}, the moduli space of $\theta$-semistable $\underline{n}$-framed representations of $Q$ with dimension $\underline{d}$ can be defined as follows:

Let $\widehat{Q}$ be the framed quiver of $Q$ which is derived from $Q$ by adding an additional vertex $i_0$ and $n_i$ arrows from $i_0$ to $i\in Q_0$. We extend the dimension vector $\underline{d}$ to a dimension vector $\underline{\hat{d}}$ of $\widehat{Q}$ by adding the entry $1$ at vertex $i_0$. Assume that $\theta$ is normalized, i.e., $\theta(\underline{d})=0$. We then also extend the stability $\theta$ to a stability $\hat{\theta}$ of $\widehat{Q}$ by adding the entry $1$ for the vertex $i_0$. Then according to \cite{ER}, the moduli space $M_{\underline{d},\underline{n}}^{\theta,\text{fr}}(Q)$ of $\theta$-semistable $\underline{n}$-framed representations of $Q$ with dimension $\underline{d}$ is isomorphic to the moduli space $M_{\underline{\hat{d}}}^{\hat{\theta}-sst}(\widehat{Q})$ of $\hat{\theta}$-semistable representations of $\widehat{Q}$ with dimension $\underline{\hat{d}}$.

We set $\underline{n}=e_j$ for the quiver associated to $(\PP^2,\lne)$, set $\underline{n}=2e_j$ for the quiver associated to $(\PP^2,\conic)$, and set $\underline{n}=e_{j_1}+e_{j_2}$ for quivers associated to $(F_n,C_n+f)$. We further use $M_{m,d}^{L,\text{fr}}$, $M_{m,d}^{C,\text{fr}}$, $M_{m,d_1,d_2}^{F_n,\text{fr}}$ to denote the corresponding $\underline{n}$-framing moduli spaces. Note that here the stability $\theta$ is always assumed to be $\theta(\cdot)=\{\bard,\cdot\}$ for a given dimension vector $\bard$. We are ready to prove Theorem \ref{thm:f/uf}.

\begin{proof}[Proof of Theorem \ref{thm:f/uf}]
Let us start from case (1): $M_{m,d}^{L,fr}\simeq M_{m+1,d}^L$ if $m$ divides $d$. This case is first established by Reineke and Weist in \cite{RW21} by comparing the stability conditions on both sides. Actually, one can easily check that both $M_{m,d}^{L,fr}$ and $M_{m+1,d}^L$ have the same underlying quiver and dimension vector $\sum_{k=0}^{m}e_{i_k}+de_j$. Here we use $i_0,\cdots,i_m$ to denote
the left $m+1$ vertices for the corresponding quiver of $M_{m+1,d}^L$. The stability for $M_{m,d}^{L,fr}$ is given by
$$\hat{\theta}=\sum_{k=0}^{m} e_{i_k}^*-\frac{m}{d}e_j^*$$
while the stability for $M_{m+1,d}^L$ is 
$$\theta'=d\left(\sum_{k=0}^{m} e_{i_k}^*\right)-(m+1)e_j^*.$$
Then we have 
\[\hat{\theta}=\frac{(m+d)}{d(m+1+d)}\theta'+\frac{1}{m+1+d}\dim\]
where $\dim=\sum e_{i_k}^*+e_j^*$ is the linear function recording the total dimension. So $M_{m,d}^{L,fr}\simeq M_{m+1,d}^L$ if $m$ divides $d$. In a similar way, we can prove case (2): $M_{m,d}^{C,fr}\simeq M_{m+1,d}^{C}$ if $m$ divides $d$; and case (3): $M_{m,d_1,d_2}^{F_0,fr}\simeq M_{m+1,d_1,d_2}^{F_0}$ if $m$ divides $d_1+d_2$. 

As for case (4): $M_{(1-n)d_1+d_2,d_1,d_2}^{F_n,fr}\simeq M_{(1-n)d_1+d_2+1,d_1,d_2}^{F_n}$ with $n>0$, we need a more careful analysis of the stability conditions on both sides. In this case, $m=(1-n)d_1+d_2\geq 0$.
First, it is easy to see that both sides have the same underlying quiver with dimension vector 
$$\sum_{k=0}^m e_{i_k}+d_1e_{j_1}+d_2e_{j_2}.$$
The stability for $M_{(1-n)d_1+d_2,d_1,d_2}^{F_n,fr}$ is 
$$\hat{\theta}=\sum_{k=0}^m e_{i_k}^*+(n-1)e_{j_1}^*-e_{j_2}^*$$
while the stability for $M_{(1-n)d_1+d_2+1,d_1,d_2}^{F_n}$ is 
\[\theta'=(d_1+d_2)\left(\sum_{k=0}^m e_{i_k}^*+(n-1)e_{j_1}^*-e_{j_2}^*\right)-(e_{j_1}^*+e_{j_2}^*).\]
Let $V$ be a $\hat{\theta}$-semistable representation. Then for all the nonzero proper subrepresentations $U\subset V$, we have 
\[\frac{\hat{\theta}(\textbf{dim}\, U)}{\dim U}\leq \frac{\hat{\theta}(\textbf{dim}\, V)}{\dim V}=\frac{1}{\dim V}\]
where we use $\textbf{dim}\,U$, $\textbf{dim}\,V$ to denote the corresponding dimension vectors and use $\dim U$, $\dim V$ to denote the corresponding total dimensions. Since $\frac{\dim U}{\dim V}<1$ as $U$ is a proper subrepresentation and $\hat{\theta}(\textbf{dim}\, U)$ must be an integer, the above equality is equivalent to
\[\hat{\theta}(\textbf{dim}\, U)\leq 0.\]
Next if  $V$ is a $\theta'$-semistable representation, then we have 
\[\theta'(\textbf{dim}\,U)=(d_1+d_2)\hat{\theta}(\textbf{dim}\,U)-(e_{j_1}^*+e_{j_2}^*)(\textbf{dim}\,U)\leq 0\]
for all the nonzero proper subrepresentations $U\subset V$. As $U$ is proper and $\hat{\theta}(\textbf{dim}\,U)$ is an integer, the above equality is also equivalent to 
\[\hat{\theta}(\textbf{dim}\, U)\leq 0.\]
We may conclude that $V$ is $\hat{\theta}$-semistable if and only if it is $\theta'$-semistable. Then case (4): $M_{(1-n)d_1+d_2,d_1,d_2}^{F_n,fr}\simeq M_{(1-n)d_1+d_2+1,d_1,d_2}^{F_n}$ follows.

\end{proof}

Next, we are going to prove Theorem \ref{thm:smres}. 
According to \cite[Theorem 4.3]{R17}, for a quiver $Q$ with dimension vector $\underline{d}$ and stabilities $\theta$ and $\theta'$, if we assume that 
\begin{enumerate}
\item[(I)] the moduli $M_{\underline{d}}^{\theta-st}(Q)$ of $\theta$-stable representations is not empty;
\item[(II)] $\theta(\underline{d})=0$ and $\theta'$ is a generic deformation of $\theta'$;
\item[(III)] the restriction of Euler form $\langle\cdot,\cdot\rangle$ to $\Ker(\theta)$ is symmetric.
\end{enumerate}
then there exists a natural small resolution $p:M_{\underline{d}}^{\theta'-sst}(Q)\rightarrow M_{\underline{d}}^{\theta-sst}(Q)$. We have the following criteria \cite[Theorem 5.1]{R17} to check whether $\theta'$ is a generic deformation of $\theta$: 

We further assume $\underline{d}=(d_i)_{i\in Q_0}$ to be indivisible, i.e., $\gcd(d_i:i\in Q_0)=1$. Then there exists a stability $\eta$ such that $\eta(\underline{d})=0$ and $\eta(\underline{e})\neq 0$ for any $\underline{e}\in \N Q_0$ with $0\neq \underline{e}\lneqq \underline{d}$ and $\theta(\underline{e})=0$. If there exist a constant $C\in \N$ with 
\[C>\max(\max(\eta(\underline{e}):\underline{e}\leq \underline{d},\theta(\underline{e})<0),\max(-\eta(\underline{e}):\underline{e}\leq \underline{d},\theta(\underline{e})>0))\]
such that $\theta'=C\theta+\eta$, then $\theta'$ is a generic deformation of $\theta$. We are now ready to prove Theorem \ref{thm:smres}.

\begin{proof}[Proof of Theorem \ref{thm:smres}]
Recall that stabilities $\theta$ of $M_{m,d}^{L}$, $M_{m,d}^{C}$, $M_{m,d_1,d_2}^{F_0}$ and $M_{(1-n)d_1+d_2,d_1,d_2}^{F_n}$ are always given by $\{\underline{d},\cdot\}$ where $\underline{d}$ are the corresponding dimension vectors. It is then easy to check that the restriction of Euler form $\langle\cdot,\cdot\rangle$ to $\Ker(\theta)$ is symmetric, i.e., condition (III) is always satisfied. It is also easy to check that condition (I) is always satisfied. This is due to the extra conditions we put for different cases of Theorem \ref{thm:smres}, e.g. we require that $m$ divides $d$ and $m>d$ for case (1). So we only need to check condition (II). For this, we will use the above criteria. 

Let us start from case (1). In this case, the stability condition $\theta'$ for $M_{m-1,d}^{L,fr}$ is given by
\[\theta'=(r+1)d\hat{\theta}-\dim\]
where we set $r=\frac{m}{d}$ and
\[\hat{\theta}=e_{i_0}^*+d\left(\sum_{k=1}^{m-1}e_{i_k}^*\right)-(rd-1)e_j^*.\]
Note that similar to the proof of Theorem \ref{thm:f/uf}, we use $i_0,\cdots,i_{m-1}$ to denote
the left $m$ vertices for the corresponding quiver of $M_{m-1,d}^{L,fr}$. The stability $\theta$ for $M_{m,d}^{L}$ can be explicit written down:
\[\theta=\sum_{k=0}^{m-1}e_{i_k}^*-re_j^*\]
Then we have 
\[\theta'=C\theta+\eta\]
where $C=(r+1)d^2-1$ and 
\[\eta=-(r+1)(d-1)(de_{i_0}^*-e_j^*).\]
Note that at least one entry of $\underline{d}=(d_i)_{i\in Q_0}$ is $1$. So $\underline{d}$ is indivisible. It is also quite quite easy to verify that the conditions for the criteria are satisfied. So $\theta'$ is a generic deformation of $\theta$. Note that the stability conditions $\theta$ and $\theta'$ for the corresponding quiver in case (2) of Theorem \ref{thm:smres} have the same expression as in case (1). So condition (II) is also satisfied for case (2). The expression of $\theta$ and $\theta'$ for the corresponding quiver in case (3) is quite similar. Actually, we only need to replace $d$ by $d_1+d_2$ and $e_j^*$ by $e_{j_1}^*+e_{j_2}^*$ in the expression of $\theta$, $\theta'$ for case (1). So condition (II) is also satisfied for case (3) of Theorem \ref{thm:smres}.

To verify condition (II) for case (4), we need a more careful analysis of the stability $\theta'$ for $M_{(1-n)d_1+d_2-1,d_1,d_2}^{F_n,fr}$. First, by an easy computation, we know that the stability $\tilde{\theta}$ for the unframed moduli $M_{(1-n)d_1+d_2-1,d_1,d_2}^{F_n}$ is 
\[\tilde{\theta}=(d_1+d_2)\left(\sum_{k=1}^{m-1}e_{i_k}^*+(n-1)e_{j_1}^*-e_{j_2}^*\right)+(e_{j_1}^*+e_{j_2}^*)\]
where we set $m=(1-n)d_1+d_2$. It is not hard to see that $\tilde{\theta}$ is equivalent to the stability
\[\sum_{k=1}^{m-1}e_{i_k}^*+(n-1)e_{j_1}^*-e_{j_2}^*.\]
The latter can be further normalized as 
\[\bar{\theta}=(m+d_1+d_2-1)\left(\sum_{k=1}^{m-1}e_{i_k}^*+(n-1)e_{j_1}^*-e_{j_2}^*\right)+\left(\sum_{k=1}^{m-1}e_{i_k}^*+e_{j_1}^*+e_{j_2}^*\right).\]
So the stability for framed moduli $M_{(1-n)d_1+d_2-1,d_1,d_2}^{F_n,fr}$ is given by
\[\hat{\theta}=e_{i_0}^*+\bar{\theta}\]
which can be further normalized as
\[\theta'=(m+d_1+d_2)\left(\sum_{k=1}^{m-1}e_{i_k}^*+(n-1)e_{j_1}^*-e_{j_2}^*\right)+\dim.\]
where $\dim=\sum_{k=0}^{m-1}e_{i_k}^*+e_{j_1}^*+e_{j_2}^*$. Recall that the stability $\theta$ for $M_{(1-n)d_1+d_2,d_1,d_2}^{F_n}$ is given by $\{\underline{d},\cdot\}$ which can be normalized as
\[\theta=\sum_{k=0}^{m-1}e_{i_k}^*+(n-1)e_{j_1}^*-e_{j_2}^*.\]
So we have 
\[\theta'=C\theta+\eta\]
where $C=(m+d_1+d_2)$ and $\eta=\dim-(m+d_1+d_2)e_{i_0}^*$. Using these explicit expressions for $C$ and $\eta$, it is easy to verify that $\theta'$ is a generic deformation of $\theta$.
\end{proof}

\section{BPS invariants}\label{section:bps}
In this section, we prove Theorem \ref{thm:bps} and Corollary \ref{cor:bpsca}.

\begin{proof}[Proof of Theorem \ref{thm:bps}]
Combining Theorems \ref{thm:loc/rel} and \ref{thm:GW/Kr}, we know that for the following pairs $(X,D)$:
\[(\PP^2,\lne),\quad (\PP^2,\conic),\quad (F_n,C_n+f),\,n\geq 0\]
we have 
\begin{equation}\label{eqn:pfbps1}
F_{\beta}^{\cO_X(-D)}=\Omega_{M_{\beta}^{X/D}}(q)\frac{(2\sin(h/2))^{\Tg\cdot\beta-1}}{2\sin(\frac{(D\cdot\beta)h}{2})}
\end{equation}
if $\Tg\cdot\beta>0$, where $q=e^{ih}$. We recall that $M_{\beta}^{X/D}$ can be determined from $(X,D)$ and $\beta$ as follows:
\begin{enumerate}
\item[(a)] if $(X,D)=(\PP^2,\lne)$ and $\beta=d[l]$, then $M_{\beta}^{X/D}=M^{L}_{2d,d}$ where $[l]$ is the line class;
\item[(b)] if $(X,D)=(\PP^2,\conic)$ and $\beta=d[l]$, then $M_{\beta}^{X/D}=M^{C}_{d,d}$;
\item[(c)] if $(X,D)=(F_n,C_n+f)$ and $\beta=d_1C_{-n}+d_2f$, then $M_{\beta}^{X/D}=M^{F_n}_{m,d_1,d_2}$ where $m=\Tg\cdot\beta=(1-n)d_1+d_2$.
\end{enumerate}
When the stable locus of $M_{\beta}^{X/D}$ is nonempty, then by Theorem \ref{thm:smres} we know that there is a small resolution of $M_{\beta}^{X/D}$ by a framed moduli. Such a framed moduli is a $\PP^{D\cdot\beta-1}$-bundle over the smooth quiver moduli $M^{\cO_X(-D)}_{\beta}$. The latter can be determined from $(X,D)$ and $\beta$ as follows:
\begin{enumerate}
\item[(i)] if $(X,D)=(\PP^2,\lne)$ and $\beta=d[l]$, then $M_{\beta}^{\cO_X(-D)}=M^{L}_{2d-1,d}$;
\item[(ii)] if $(X,D)=(\PP^2,\conic)$ and $\beta=d[l]$, then $M_{\beta}^{\cO_X(-D)}=M^{C}_{d-1,d}$;
\item[(iii)] if $(X,D)=(F_n,C_n+f)$ and $\beta=d_1C_{-n}+d_2f$, then $M_{\beta}^{\cO_X(-D)}=M^{F_n}_{m-1,d_1,d_2}$ where $m=\Tg\cdot\beta=(1-n)d_1+d_2$.
\end{enumerate}
So by Leray–Hirsch theorem, we have
\begin{equation}\label{eqn:pfbps2}
\Omega_{M_{\beta}^{X/D}}=P_{\PP^{D\cdot\beta-1}}\Omega_{M^{\cO_X(-D)}_{\beta}}
\end{equation}
where 
$$P_{\PP^{D\cdot\beta-1}}=(-1)^{D\cdot\beta-1}\frac{q^{\frac{D\cdot\beta}{2}}-q^{-\frac{D\cdot\beta}{2}}}{q^{1/2}-q^{-1/2}}=(-1)^{D\cdot\beta-1}\frac{2\sin(\frac{(D\cdot\beta)h}{2})}{2\sin(h/2)}$$
is the shifted Poincar\'e polynomial of $\PP^{D\cdot\beta-1}$. When the stable locus of $M_{\beta}^{X/D}$ is empty, one can easily check that the stable locus of $M^{\cO_X(-D)}_{\beta}$ is empty also. So equation \eqref{eqn:pfbps2} still holds.

Combining \eqref{eqn:pfbps1} and \eqref{eqn:pfbps2}, we have
\begin{equation}\label{eqn:GW/DT}
\frac{F_{\beta}^{\cO_X(-D)}}{(2\sin(h/2))^{\Tg\cdot\beta-2}}=(-1)^{D\cdot\beta-1}\Omega_{M^{\cO_X(-D)}_{\beta}}.
\end{equation}
By \eqref{eqn:bps}, the left-hand side equals to $\sum_{g} n_{g,\beta}^{\mathcal{O}_X(-D)}(2\sin(h/2))^{2g}$.
So
\begin{equation}\label{eqn:pfbps3}
\sum_{g} n_{g,\beta}^{\mathcal{O}_X(-D)}(2\sin(h/2))^{2g}=(-1)^{D\cdot\beta-1}\Omega_{M^{\cO_X(-D)}_{\beta}}.
\end{equation}
Finally, for those pairs $(X,D)=(F_n,C_n+(s+1)f)$ with $s\geq 1$ and $\beta=d_1C_{-n}+d_2f$, we set 
$M_{\beta}^{\cO_X(-D)}=M^{F_{n+2s}}_{m-1,d_1,d_2+sd_1}$ with $m=\Tg\cdot\beta=(1-n-s)d_1+d_2$. Then according to the deformation equivalence discussed in Section \ref{subsec:definv}, we know that \eqref{eqn:pfbps3} still holds.
\end{proof}

Next, we will give a proof of Corollary \ref{cor:bpsca} based on Theorem \ref{thm:bps}.
\begin{proof}[Proof of Corollary \ref{cor:bpsca}]
As one can easily check that $\Omega_{M^{\cO_X(-D)}_{\beta}}$ only consists of stable representations, the condition that $\Omega_{M^{\cO_X(-D)}_{\beta}}$ is not empty implies that 
\[\Omega_{M^{\cO_X(-D)}_{\beta}}(q)=(-q^{1/2})^{-\dim_{\C}M^{\cO_X(-D)}_{\beta}}\sum_i \dim H^i(M^{\cO_X(-D)}_{\beta},\Q)(-q^{1/2})^i.\]
Note that in this case, $\Omega_{M^{\cO_X(-D)}_{\beta}}$ is smooth. So the intersection cohomology of $M^{\cO_X(-D)}_{\beta}$ coincides with the usual cohomology.

Now Corollary \ref{cor:bpsca} follows from the facts that $\dim_{\C}M^{\cO_X(-D)}_{\beta}=1-<\underline{d},\underline{d}>$ where $\underline{d}$ is the dimension vector of $M^{\cO_X(-D)}_{\beta}$, and $M^{\cO_X(-D)}_{\beta}$ is connected, so $\dim H^0(M^{\cO_X(-D)}_{\beta},\Q)=1$.
\end{proof}

\section{Proof of Theorem \ref{thm:main}}\label{sec:proof}
We prove Theorem \ref{thm:main} in this section. First, by the deformation equivalence discussed in Section \ref{subsec:definv}, we only need to consider the following pairs $(X,D)$:
\[(\PP^2,\lne),\quad (\PP^2,\conic),\quad (F_n,C_n+f),\,n\geq 0.\]

Next, let us translate the recursion formula \eqref{eqn:recurfm} into a recursion formula of quiver DT-invariants by \eqref{eqn:GW/DT}:
\begin{equation}\label{eqn:qDTfm}
\Omega_{M^{\cO_X(-D)}_{\beta}}=\sum_{\beta_1+\beta_2=\beta\atop \beta_1,\beta_2>0} \Omega_{M^{\cO_X(-D)}_{\beta_1}}\Omega_{M^{\cO_X(-D)}_{\beta_2}}
\left(P_{\PP^{D\cdot\beta_1-1}}\right)^2
{\Tg\cdot\beta-3\choose \Tg\cdot\beta_1-1}.
\end{equation}
Here we recall that 
\[P_{\PP^{D\cdot\beta_1-1}}=(-1)^{D\cdot\beta_1-1}\frac{q^{\frac{D\cdot\beta_1}{2}}-q^{-\frac{D\cdot\beta_1}{2}}}{q^{1/2}-q^{-1/2}}.\]
When $(X,D)=(\PP^2,\lne)$, the above recursion was first derived by Reineke-Weist \cite[Theorem 1.2]{RW21} using a formula relating DT-invariants of framed moduli to unframed one together with some geometric properties of the corresponding quiver moduli\footnote{These geometric properties are just case (1) appearing in Theorems \ref{thm:dual}, \ref{thm:f/uf} and \ref{thm:smres}.}. Actually, we will show in the proof that their proof can be generalized to the cases $(\PP^2,\conic)$ and $(F_n,C_n+f)$ as well.
\begin{proof}[Proof of Theorem \ref{thm:main}]
Let $Q$ be a quiver with stability $\theta$. We use $\Lambda_0^{+}$ to denote the set of nonzero dimension vectors $\underline{d}$ such that $\theta(\underline{d})=0$. The key relation used in Reineke-Weist's proof can be described as follows:
\begin{equation}\label{eqn:keyrl}
1+\sum_{\underline{d}\in \Lambda_0^{+}}\Omega_{M_{\underline{d},\underline{n}}^{\theta,\text{fr}}}(-1)^{\underline{n}\cdot\underline{d}}x^{\underline{d}}=\Exp\left(\sum_{\underline{d}\in \Lambda_0^+}P_{\PP^{\underline{n}\cdot\underline{d}-1}}\Omega_{M_{\underline{d}}^{\theta-sst}}(-1)^{\underline{n}\cdot\underline{d}}x^{\underline{d}}\right)
\end{equation}
where $\Exp(\cdot)$ is the plethystic exponential\footnote{For a formal power series $f(x)$ without constant constant, $\Exp(f)=\exp(\sum_{k=1}^{\infty}\frac{f(x^k)}{k})$.}. Note that we need a technical assumption to make the above equality holds: the restriction of the Euler form $\langle\cdot,\cdot\rangle$ to $\Lambda_0^+$ is symmetric. This assumption always holds for all the cases we will discuss in the next. As before, for a given dimension vector $\underline{d}$, we always choose the stability $\theta$ to be $\{\underline{d},\cdot\}$.

Case I: $(X,D)=(\PP^2,\lne)$. We choose the quivers to be those corresponds to the pair $(\PP^2,\lne)$, and set $\underline{d}=\sum_{k=1}^{2d}e_{i_k}+de_j$, $\underline{n}=e_j$. In this case, $\theta=\sum_{k=1}^{2d} e_{i_k}^*-2e_j^*$.
Then using the previous notation, we have 
\[M_{\underline{d},\underline{n}}^{\theta,\text{fr}}=M_{2d,d}^{L,\text{fr}},\quad M_{\underline{d}}^{\theta-sst}=M_{2d,d}^L.\]
By Theorems \ref{thm:dual}, \ref{thm:f/uf} and \ref{thm:smres}, we have
\[M_{2d,d}^{L,\text{fr}}\simeq M_{2d+1,d}^L\simeq M_{2d+1,d+1}^L,\quad \Omega_{M_{2d,d}^L}=P_{\PP^{d-1}}\Omega_{M_{2d-1,d}^L}.\]
We set $z_d^{L}(q)=\Omega_{M_{2d-1,d}^L}(q)$. Then by \eqref{eqn:keyrl}, we may deduce that 
\[z_{d+1}^L=\sum_{a_1+a_2\cdots+a_d=d\atop a_i\geq 0}\frac{(2d)!}{\prod_{k=1}^d \big((2k)!\big)^{a_k}(a_k)!}\prod_{k=1}^d\left(\big(P_{\PP^{k-1}}\big)^2 z_k^L\right)^{a_k}\]
which is equivalent to 
\[\frac{z^L_{d+1}}{(2d)!}=\sum_{a_1+a_2\cdots+a_d=d\atop a_i\geq 0}\frac{1}{(a_k)!}\prod_{k=1}^d\left(\frac{\big(P_{\PP^{k-1}}\big)^2 z^L_k}{(2k)!}\right)^{a_k}.\]
So by summing over $d$, we have
\[1+\sum_{d>0}\frac{z^L_{d+1}}{(2d)!}x^d=\exp\left(\sum_{k>0}\frac{\big(P_{\PP^{k-1}}\big)^2 z^L_k}{(2k)!}x^k\right).\]
By further taking a derivative $2x\frac{d}{dx}$ on both sides, we have
\[\sum_{d>0}\frac{z^L_{d+1}}{(2d-1)!}x^d=\left(\sum_{k>0}\frac{\big(P_{\PP^{k-1}}\big)^2 z^L_k}{(2k-1)!}x^k\right)\left(\sum_{d\geq 0}\frac{z^L_{d+1}}{(2d)!}x^d\right).\]
Here we have used the fact that $z^L_1=1$ which can be deduced from the fact that $M_{1,1}^L$ is a point. So by taking the coefficients of $x^{d-1}$ on both sides, we get the recursion
\[z^L_d=\sum_{d_1+d_2=d\atop d_1,d_2>0}z^L_{d_1}z^L_{d_2}\big(P_{\PP^{d_1-1}}\big)^2{2d-3\choose 2d_1-1}.\]
This is exactly equation \eqref{eqn:qDTfm} for the pair $(\PP^2,\lne)$ by noting that $z^L_d=\Omega_{M_{d[l]}^{\cO_{\PP^2}(-1)}}$.

Case II: $(X,D)=(\PP^2,\conic)$. In this case, the quivers correspond to those of the pair $(\PP^2,\conic)$, and set $\underline{d}=\sum_{k=1}^{d}e_{i_k}+de_j$, $\underline{n}=e_j$. In this case, $\theta=\sum_{k=1}^{d} e_{i_k}^*-e_j^*$. Then using case (2) of Theorems \ref{thm:dual}, \ref{thm:f/uf} and \ref{thm:smres}, we have
\[M_{d,d}^{C,fr}\simeq M_{d+1,d+2}^C,\quad \Omega_{M_{d,d}^C}=P_{\PP^{2d-1}}\Omega_{M_{d-1,d}^C}.\]
In this case, we may deduce from \eqref{eqn:keyrl} that
\[1+\sum_{d>0}\frac{z_{d+2}^C}{d!}x^d=\exp\left(\sum_{k>0}\frac{\big(P_{\PP^{2k-1}}\big)^2 z^C_k}{(k)!}x^k\right)\]
with $z_d^C=\Omega_{M_{d-1,d}^C}$. By taking a derivative $x\frac{d}{dx}$ on both sides, we get the recursion
\[z_{d}^C=\sum_{d_1+d_2=d\atop d_1,d_2>0}z_{d_1}^Cz_{d_2}^C P_{\PP^{2d_1-1}}{d-3\choose d_1-1}\]
which is exactly equation \eqref{eqn:qDTfm} for the pair $(\PP^2,\conic)$. Here we have used the fact that $z_2^C=1$ which can be similarly deduced as in case (I).

Case III: $(X,D)=(F_n,C_n+f),\,n\geq 0$. We then choose the quivers to be those correspond to the pair $(F_n,C_n+f)$. For the curve class $\beta=d_1C_{-n}+d_2f$ with $\Tg\cdot\beta=(1-n)d_1+d_2\geq 0$, we then set the dimension vector $\underline{d}=\sum_{k=1}^{(1-n)d_1+d_2}e_{i_k}+d_1e_{j_1}+d_2e_{j_2}$ and $\underline{n}=e_{j_1}+e_{j_2}$. The stability $\theta=\sum_{k=1}^{(1-n)d_1+d_2}e_{i_k}^*+(n-1)e_{j_1}^*-e_{j_2}^*$.
In this case, we have
\[M_{\underline{d},\underline{n}}^{\theta,\text{fr}}=M_{(1-n)d_1+d_2,d_1,d_2}^{F_n,\text{fr}},\quad M_{\underline{d}}^{\theta-sst}=M_{(1-n)d_1+d_2,d_1,d_2}^{F_n}.\]
Then by cases (3), (4) of Theorems \ref{thm:dual}, \ref{thm:f/uf} and \ref{thm:smres}, we have
\[M_{(1-n)d_1+d_2,d_1,d_2}^{F_n,\text{fr}}\simeq M_{(1-n)d_1+d_2+1,d_1+1,d_2+n+1}^{F_n},\, \Omega_{M_{(1-n)d_1+d_2,d_1,d_2}^{F_n}}=P_{\PP^{d_1+d_2-1}}\Omega_{M_{(1-n)d_1+d_2-1,d_1,d_2}^{F_n}}.\]
After plugging them into equation \eqref{eqn:keyrl} and setting $z_{d_1,d_2}^{F_n}=\Omega_{M_{(1-n)d_1+d_2-1,d_1,d_2}^{F_n}}$, we may deduce that 
\[1+\sum_{(1-n)d_1+d_2\geq 0 \atop d_1+d_2>0}\frac{z_{d_1+1,d_2+n+1}^{F_n}}{\big((1-n)d_1+d_2\big)!}x_1^{d_1}x_2^{d_2}=G\exp\left(\sum_{(1-n)k_1+k_2>0}\frac{\big(P_{\PP^{d_1+d_2-1}}\big)^2z_{k_1,k_2}^{F_n}}{\big((1-n)k_1+k_2\big)!}x_1^{k_1}x_2^{k_2}\right)\]
where 
\[G=1+\sum_{(1-n)d_1+d_2= 0 \atop d_1+d_2>0}z_{d_1+1,d_2+n+1}^{F_n}x_1^{d_1}x_2^{d_2}.\]
After taking a derivative $(1-n)x_1\frac{\partial}{\partial x_1}+x_2\frac{\partial}{\partial x_2}$ on both sides and noting that
\[\left((1-n)x_1\frac{\partial}{\partial x_1}+x_2\frac{\partial}{\partial x_2}\right)G=0,\]
we may deduce that 
\[z_{d_1,d_2}^{F_n}=\sum_{k_1+k_1'=d_1\atop k_2+k_2'=d_2}z_{k_1,k_2}^{F_n}z_{k_1',k_2'}^{F_n}\big(P_{\PP^{k_1+k_2-1}}\big)^2{(1-n)d_1+d_2-3\choose (1-n)k_1+k_2-1}.\]
Here we have used the fact that $z_{1,n+1}^{F_n}=1$. This is exactly equation \eqref{eqn:qDTfm} for the pair $(F_n,C_n+f)$.
\end{proof}

\section{Numerical results}\label{sec:numdata}
In this section, we list the first few $F_{\beta}^{\cO_X(-D)}$ and $n_{g,\beta}^{\mathcal{O}_X(-D)}$
for the following five pairs $(X,D)$:
\[(\PP^2,\lne),\quad (\PP^2,\conic),\quad (F_0,C_0+f),\quad (F_1,C_1+f),\quad (F_2,C_2+f)\]
based on the recursion \eqref{eqn:recurfm}. The initial data can be determined from the quiver side (combining \eqref{eqn:defbps} and \eqref{eqn:bps}).

Case I: $(X,D)=(\PP^2,\lne)$. In this case, curve classes $\beta$ of $\PP^2$ are indexed by integers $d\in \N$. We only need to compute those $F_{\beta}^{\cO_X(-D)}$ such that $0<\Tg\cdot\beta=2d<3$, i.e., $d=1$. When $d=1$, $F_{d}^{\mathcal{O}_{\mathbb{P}^2}(-1)}=1$ can be deduced from \eqref{eqn:defbps}, \eqref{eqn:bps} and the fact that $M_{\beta}^{\cO_X(-D)}=M_{1,1}^L$ is a point. We have the following table for the generating series of Gromov--Witten invariants of $\mathcal{O}_{\mathbb{P}^2}(-1)$ up to degree $5$:


\begin{table}[htb]  
\begin{center}   
\caption{GW-invariants of $\mathcal{O}_{\mathbb{P}^2}(-1)$}
\begin{tabular}{|c|c|}  
\hline $d$ & $F_{d}^{\mathcal{O}_{\mathbb{P}^2}(-1)}$\\
\hline $1$ & \footnotesize $1$\\
\hline $2$ & \footnotesize $(q-1)^2/q$\\
\hline $3$ & \footnotesize $(q - 1)^4(q^2 + 5q + 1)/q^3$\\
\hline $4$ & \footnotesize $ (q - 1)^6(q^6 + 7q^5 + 29q^4 +
64q^3 + 29q^2 + 7q + 1)/q^6$\\
\hline $5$ & \tabincell{c}{\footnotesize $(q - 1)^8(q^{12} + 9q^{11} + 46q^{10} + 175q^9 +
506q^8 + 1138q^7$ \\ \footnotesize $ + 1727q^6 + 1138q^5 + 506q^4 + 175q^3 + 46q^2 + 9q + 1)/q^{10}$}\\
\hline
\end{tabular}
\end{center}   
\end{table}

The BPS invariants $n_{g,d}^{\mathcal{O}_{\mathbb{P}^2}(-1)}$ are determined from the following formula
\[F_{d}^{\mathcal{O}_{\mathbb{P}^2}(-1)}=\sum_{g\geq 0} n_{g,d}^{\mathcal{O}_{\mathbb{P}^2}(-1)}(2\sin(h/2))^{2g-2+2d}\,,\]
where $q=e^{\sqrt{-1}h}$.

\begin{table}[htb]
\begin{center}   
\caption{ BPS invariants for $\cO_{\PP^2}(-1)$ }
\begin{tabular}{|l|c|c|c|c|c|c|c|}  
\hline 
\normalsize\diagbox{$d$}{$g$} & $0$& $1$ & $2$ & $3$ & $4$ &  $5$ &  $6$\\
\hline $1$ & $1$& $0$ &  $0$ & $0$ & $0$ &$0$ & $0$\\
\hline $2$ & $-1$& $0$ &$0$ & $0$ & $0$ &$0$ & $0$ \\
\hline $3$ & $7$& $-1$ &$0$ & $0$ & $0$ &$0$ & $0$ \\
\hline $4$ & $-138$ & $66$ & $-13$ & $1$ & $0$ &$0$ & $0$ \\
\hline $5$ & $5477$& $-5734$ &$3031$ & $-970$ & $190$ &$-21$ & $1$ \\
\hline
\end{tabular}
\end{center}   
\end{table}

Case II: $(X,D)=(\PP^2,\conic)$. Curve classes $\beta$ are again indexed by integers $d\in\N$. In this case, we need to compute those $F_{d}^{\mathcal{O}_{\mathbb{P}^2}(-2)}$ such that $d=1,2$ which can be similarly determined from the quiver side. 
\begin{table}[htb]  
\begin{center}   
\caption{GW-invariants of $\mathcal{O}_{\mathbb{P}^2}(-2)$}
\begin{tabular}{|c|c|}  
\hline $d$ & $F_{d}^{\mathcal{O}_{\mathbb{P}^2}(-2)}$\\
\hline $1$ & \footnotesize $-(-q)^{1/2}/(q-1)$\\
\hline $2$ & \footnotesize $-1$\\
\hline $3$ & \footnotesize $(q-1)(q^2+2q+1)/(-q)^{3/2}$\\
\hline $4$ & \footnotesize $(q-1)^2(q^6+3q^5+7q^4+10q^3+7q^2+3q+1)/q^4$\\
\hline $5$ & \tabincell{c}{\footnotesize $-(q-1)^3(q^{12}+4q^{11}+11q^{10}+25q^9+46q^8+71q^7$\\ \footnotesize $+84q^6+71q^5+46q^4+25q^3+11q^2+4q+1)/(-q)^{15/2}$}\\
\hline
\end{tabular}
\end{center}   
\end{table}

The BPS invariants $n_{g,d}^{\mathcal{O}_{\mathbb{P}^2}(-2)}$ are determined from the following formula
\[F_{d}^{\mathcal{O}_{\mathbb{P}^2}(-2)}=\sum_{g\geq 0} n_{g,d}^{\mathcal{O}_{\mathbb{P}^2}(-2)}(2\sin(h/2))^{2g-2+d}\,.\]

\begin{table}[htb]  
\begin{center}   
\caption{BPS invariants for $\cO_{\PP^2}(-2)$}
\begin{tabular}{|l|c|c|c|c|c|c|c|}  
\hline 
\diagbox{$d$}{$g$} & $0$& $1$ &$2$ & $3$ & $4$ &$5$ & $6$ \\
\hline $1$ & $-1$& $0$ &$0$ & $0$ & $0$ &$0$ & $0$  \\
\hline $2$ & $-1$& $0$ &$0$ & $0$ & $0$ &$0$ & $0$  \\
\hline $3$ & $-4$& $1$ &$0$ & $0$ & $0$ &$0$ & $0$  \\
\hline $4$ & $-32$ & $28$ & $-9$ & $1$ & $0$ &$0$ & $0$ \\
\hline $5$ & $-400$& $792$ &$-721$ & $365$ & $-105$ &$16$ & $-1$ \\
\hline
\end{tabular}
\end{center}   
\end{table}

Case III: $(X,D)=(F_0,C_0+f)$. In this case, $F_0$ is just $\PP^1\times\PP^1$ and 
each curve class $\beta$ can be uniquely written as $d_1C_0+d_2f$ where $d_1,d_2\in\N$. We denote $F_{\beta}^{\cO_X(-D)}$  by $F_{d_1,d_2}^{\mathcal{O}_{\mathbb{P}^1\times\mathbb{P}^1}(-1,-1)}$. By symmetry, we have $F_{d_1,d_2}^{\mathcal{O}_{\mathbb{P}^1\times\mathbb{P}^1}(-1,-1)}=F_{d_2,d_1}^{\mathcal{O}_{\mathbb{P}^1\times\mathbb{P}^1}(-1,-1)}$, and so we only need to determine those pairs $(d_1,d_2)$ such that $d_2\geq d_1$. For the initial data, we need to determine those $F_{d_1,d_2}^{\mathcal{O}_{\mathbb{P}^1\times\mathbb{P}^1}(-1,-1)}$ such that $0<\Tg\cdot\beta=d_1+d_2<3$. All of these $F_{d_1,d_2}^{\mathcal{O}_{\mathbb{P}^1\times\mathbb{P}^1}(-1,-1)}$ can be determined from the quiver side as the corresponding quiver moduli space $M_{\beta}^{\cO_X(-D)}$ is either a point or empty. 
\begin{table}[htb]  
\begin{center}   
\caption{GW-invariants of $\mathcal{O}_{\mathbb{P}^1\times\mathbb{P}^1}(-1,-1)$}
\begin{tabular}{|c|c|}  
\hline $(d_1,d_2)$ & $F_{d_1,d_2}^{\mathcal{O}_{\mathbb{P}^1\times\mathbb{P}^1}(-1,-1)}$\\
\hline $(0,1)$ & \footnotesize  $(-q)^{1/2}/(q - 1)$\\
\hline $(1,1)$ & \footnotesize  $-1$\\
\hline $(2,2)$ & \footnotesize  $(q - 1)^2(q^2 + 4q + 1)/q^2$\\
\hline $(2,3)$ & \footnotesize  $(q - 1)^3(q^4 + 5q^3 + 12q^2 + 5q + 1)/(-q)^{7/2}$\\
\hline $(2,4)$ & \footnotesize  $ -(q - 1)^4(q^6 + 6q^5 + 17q^4 + 32q^3 + 17q^2 + 6q + 1)/q^5$\\
\hline $(3,3)$ & \footnotesize  $-(q - 1)^4(q^8 + 6q^7 + 23q^6 + 58q^5 + 94q^4 + 58q^3 + 23q^2 + 6q + 1)/q^6$\\
\hline $(3,4)$ & \tabincell{c}{\footnotesize  $(q - 1)^5(q^{12} + 7q^{11} + 30q^{10} + 93q^9 + 227q^8 + 429q^7$ \\ \footnotesize $ + 586q^6 + 429q^5 + 227q^4 + 93q^3 + 30q^2 + 7q + 1)/(-q)^{17/2}$}\\
\hline
\end{tabular}
\end{center}   
\end{table}

When $d_1=0$ and $d_2\geq 2$, we have $F_{d_1,d_2}^{\mathcal{O}_{\mathbb{P}^1\times\mathbb{P}^1}(-1,-1)}=0$. This can be deduced either from the quiver side, because the corresponding quiver moduli space is empty, or from the Gromov--Witten side, because in this case curves are mapped to fibers of $\mathbb{P}^1\times\mathbb{P}^1$ and so cannot pass through $d_2\geq 2$ points in general position. When $d_1=1$, one deduces from the recursion that $$F_{1,d_2}^{\mathcal{O}_{\mathbb{P}^1\times\mathbb{P}^1}(-1,-1)}=-\left(\frac{1-q}{(-q)^{1/2}}\right)^{d_2-1}\,.$$

The BPS invariants $n_{g,(d_1,d_2)}^{\mathcal{O}_{\mathbb{P}^1\times\mathbb{P}^1}(-1,-1)}$ are determined from the following formula
\[F_{d_1,d_2}^{\mathcal{O}_{\mathbb{P}^1\times\mathbb{P}^1}(-1,-1)}=\sum_{g \geq 0} n_{g,(d_1,d_2)}^{\mathcal{O}_{\mathbb{P}^1\times\mathbb{P}^1}(-1,-1)}(2\sin(h/2))^{2g-2+d_1+d_2}\,.\]

\begin{table}[htb]  
\begin{center}   
\caption{BPS invariants for $\cO_{\PP^1\times \PP^1}(-1,-1)$}
\begin{tabular}{|l|c|c|c|c|c|c|c|c|c|}  
\hline 
\diagbox{$(d_1,d_2)$}{$g$} & $0$& $1$ &$2$ & $3$ & $4$ & $5$& $6$ \\
\hline $(0,1)$ & $1$& $0$ &$0$ & $0$ & $0$ & $0$ & $0$ \\
\hline $(1,1)$ & $-1$ & $0$ & $0$ & $0$ & $0$ & $0$ & $0$ \\
\hline $(2,2)$ & $-6$ & $1$ & $0$ & $0$ & $0$ & $0$ & $0$\\
\hline $(2,3)$ & $24$ & $-9$ & $1$ & $0$ & $0$ & $0$ & $0$ \\
\hline $(2,4)$ & $-80$ & $50$ & $-12$ & $1$ & $0$ & $0$ & $0$ \\
\hline $(3,3)$ & $-270$ & $220$ & $-79$ & $14$ & $-1$ & $0$ & $0$ \\
\hline $(3,4)$ & $2160$ & $-2865$ & $1840$ & $-690$ & $154$ & $-19$ & $1$\\
\hline 
\end{tabular}
\end{center}   
\end{table}

Case IV: $(X,D)=(F_1,C_1+f)$. Each effective curve class $\beta$ can be written as $d_1C_{-1}+d_2f$ where $d_1,d_2\in \N$. So the generating series $F_{\beta}^{\cO_X(-D)}$ can be written as $F_{d_1,d_2}^{\mathcal{O}_{F_1}(-C_1-f)}$.
The corresponding quiver $M_{\beta}^{X/D}$ is
\begin{center}
\includegraphics[width=0.9\textwidth]{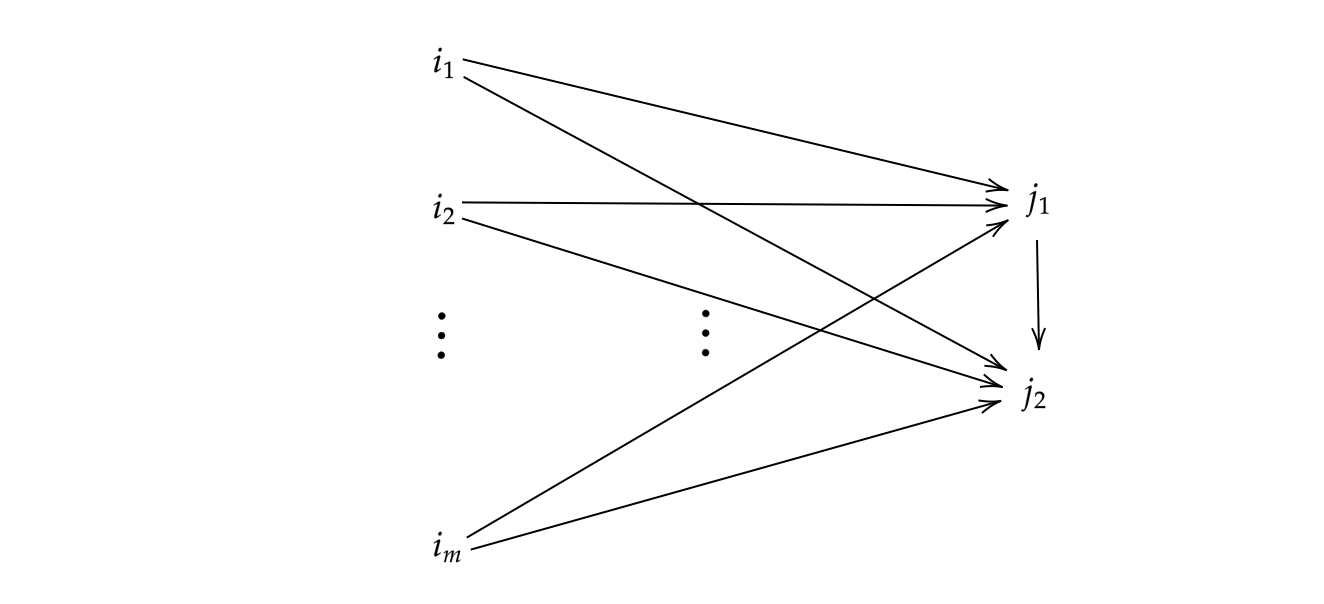}
\end{center}
If $\beta=d_1C_{-1}+d_2f$, then the number $m$ of those vertices $i_k$ on the left-hand side equals to $\Tg\cdot\beta=d_2$. The dimensions associated to vertices $j_1,j_2$ are $d_1,d_2$ respectively.
And the stability condition is
\[\theta=\sum_k i_k^*-j_2^*.\]
So if $d_1>d_2$, then for each quiver representation $(V_i,f_i)_{i\in Q_0}$, we could always find a proper subspace $W_{j_1}$ of $V_{j_1}$ which is the images from the $d_2$ one dimensional spaces $V_{i_k}$. Then we can naturally get a nonzero proper subrepresentation of $(V_i,f_i)_{i\in Q_0}$ which still has $\theta=0$ because $\theta(j_1)=0$. This implies that the moduli of $\theta$-stable representations is empty. Then DT-invariants associated to this quiver is just zero. Combining with Theorems \ref{thm:loc/rel} and \ref{thm:GW/Kr}, we may deduce that $F_{d_1,d_2}^{\mathcal{O}_{F_1}(-C_1-f)}=0$ if $d_1>d_2$. So we only need to determine those pairs $(d_1,d_2)$ such that $d_1\leq d_2$.

The initial data for $F_{d_1,d_2}^{\mathcal{O}_{F_1}(-C_1-f)}$ are
\[(d_2,d_2)=(0,1),\,(0,2),\,(1,1),\,(1,2),\,(2,2).\]
Note that the case $(0,2)$ can be ruled out similar to case III. For the rest pairs, all of those $F_{d_1,d_2}^{\mathcal{O}_{F_1}(-C_1-f)}$ can be determined from the quiver side since the corresponding quiver moduli  $M_{\beta}^{\cO_X(-D)}$ is always a point. We have the following table for $F_{d_1,d_2}^{\mathcal{O}_{F_1}(-C_1-f)}$:
\begin{table}[htb]  
\begin{center}   
\caption{GW-invariants of $\mathcal{O}_{F_1}(-C_1-f)$}
\begin{tabular}{|c|c|}  
\hline $(d_1,d_2)$ & $F_{d_1,d_2}^{\mathcal{O}_{F_1}(-C_1-f)}$\\
\hline $(0,1)$ & \footnotesize  $(-q)^{1/2}/(q - 1)$\\
\hline $(1,1)$ & \footnotesize  $-(-q)^{1/2}/(q - 1)$\\
\hline $(1,2)$ & \footnotesize  $1$\\
\hline $(2,2)$ & \footnotesize  $-1$\\
\hline $(2,3)$ & \footnotesize  $-(q-1)(q^2+3q+1)/(-q)^{3/2}$\\
\hline $(2,4)$ & \footnotesize  $ (q - 1)^2(q^4 + 4q^3 + 8q^2 + 4q + 1)/q^3$\\
\hline $(2,5)$ & \footnotesize $ -(q - 1)^3(q^6 + 5q^5 + 12q^4 + 20q^3 + 12q^2 + 5q + 1)/(-q)^{9/2}$\\
\hline $(3,3)$ & \footnotesize  $ (q-1)(q^2+2q+1)/(-q)^{3/2}$\\
\hline $(3,4)$ & \footnotesize $-(q - 1)^2(q^6 + 4q^5 + 11q^4 + 17q^3 + 11q^2 + 4q + 1)/q^4$    \\
\hline $(4,4)$ & \footnotesize $(q - 1)^2(q^6 + 3q^5 + 7q^4 + 10q^3 + 7q^2 + 3q + 1)/q^4$ \\
\hline $(4,5)$ & \tabincell{c}{\footnotesize $(q - 1)^3(q^{12} + 5q^{11} + 16q^{10} + 41q^9 + 82q^8 + 136q^7$\\  \footnotesize $+ 167q^6 + 136q^5 + 82q^4 + 41q^3 + 16q^2 + 5q + 1)/(-q)^{15/2}$ }\\
\hline $(5,5)$ & \tabincell{c}{\footnotesize $(q - 1)^3(q^{12} + 4q^{11} + 11q^{10} + 25q^9 + 46q^8 + 71q^7$ \\ \footnotesize $ + 84q^6 + 71q^5 + 46q^4 + 25q^3 + 11q^2 + 4q + 1)/(-q)^{15/2}
 $}\\
\hline
\end{tabular}
\end{center}   
\end{table}

The BPS invariants $n_{g,(d_1,d_2)}^{\mathcal{O}_{F_1}(-C_1-f)}$ can be determined from the following formula
\[F_{d_1,d_2}^{\mathcal{O}_{F_1}(-C_1-f)}=\sum_{g\geq 0} n_{g,(d_1,d_2)}^{\mathcal{O}_{F_1}(-C_1-f)}(2\sin(h/2))^{2g-2+d_2}.\]

\begin{table}[htb]  
\begin{center}   
\caption{BPS invariants for $\mathcal{O}_{F_1}(-C_1-f)$}
\begin{tabular}{|l|c|c|c|c|c|c|c|}  
\hline 
\diagbox{$(d_1,d_2)$}{$g$} & $0$& $1$ &$2$ & $3$ & $4$ & $5$& $6$  \\
\hline $(0,1)$ & $1$& $0$ &$0$ & $0$ & $0$ & $0$ & $0$ \\
\hline $(1,1)$ & $-1$ & $0$ & $0$ & $0$ & $0$ & $0$ & $0$\\
\hline $(1,2)$ & $1$ & $0$ & $0$ & $0$ & $0$ & $0$ & $0$ \\
\hline $(2,2)$ & $-1$ & $0$ & $0$ & $0$ & $0$ & $0$ & $0$ \\
\hline $(2,3)$ & $5$ & $-1$ & $0$ & $0$ & $0$ & $0$ & $0$ \\
\hline $(2,4)$ & $-18$ & $8$ & $-1$ & $0$ & $0$ & $0$ & $0$\\
\hline $(2,5)$ & $56$ & $-41$ & $11$ & $-1$ & $0$ & $0$ & $0$\\
\hline $(3,3)$ & $-4$ & $1$ & $0$ & $0$ & $0$ & $0$ & $0$ \\
\hline $(3,4)$ & $49$ & $-36$ & $10$ & $-1$ & $0$ & $0$ & $0$ \\
\hline $(3,5)$ & $-384$ & $499$ & $-293$ & $92$ & $-15$ & $1$ & $0$\\
\hline $(4,4)$ & $-32$ & $28$ & $-9$ & $1$ & $0$ & $0$ & $0$ \\
\hline $(4,5)$ & $729$ & $-1250$ & $1003$ & $-456$ & $120$ & $-17$ & $1$\\
\hline $(5,5)$ & $-400$ & $792$ & $-721$ & $365$ & $-105$ & $16$ & $-1$\\
\hline 
\end{tabular}
\end{center}   
\end{table}
\newpage

Case V: $(X,D)=(F_2,C_2+f)$. Still each curve class $\beta$ can be written as $d_1C_{-2}+d_2f$. And the corresponding $F_{\beta}^{\cO_X(-D)}$ can be written as $F_{d_1,d_2}^{\mathcal{O}_{F_2}(-C_2-f)}$. We 
only consider those $\beta$ such that $\Tg\cdot\beta=d_2-d_1>0$. The initial data in this case are those $F_{d_1,d_2}^{\mathcal{O}_{F_2}(-C_2-f)}$ such that $0<d_2-d_1<3$. If $d_2-d_1=1$, then $M_{\beta}^{\cO_X(-D)}$ is the moduli associated to following 2-Kronecker quiver:
\begin{center}
\includegraphics[width=0.9\textwidth]{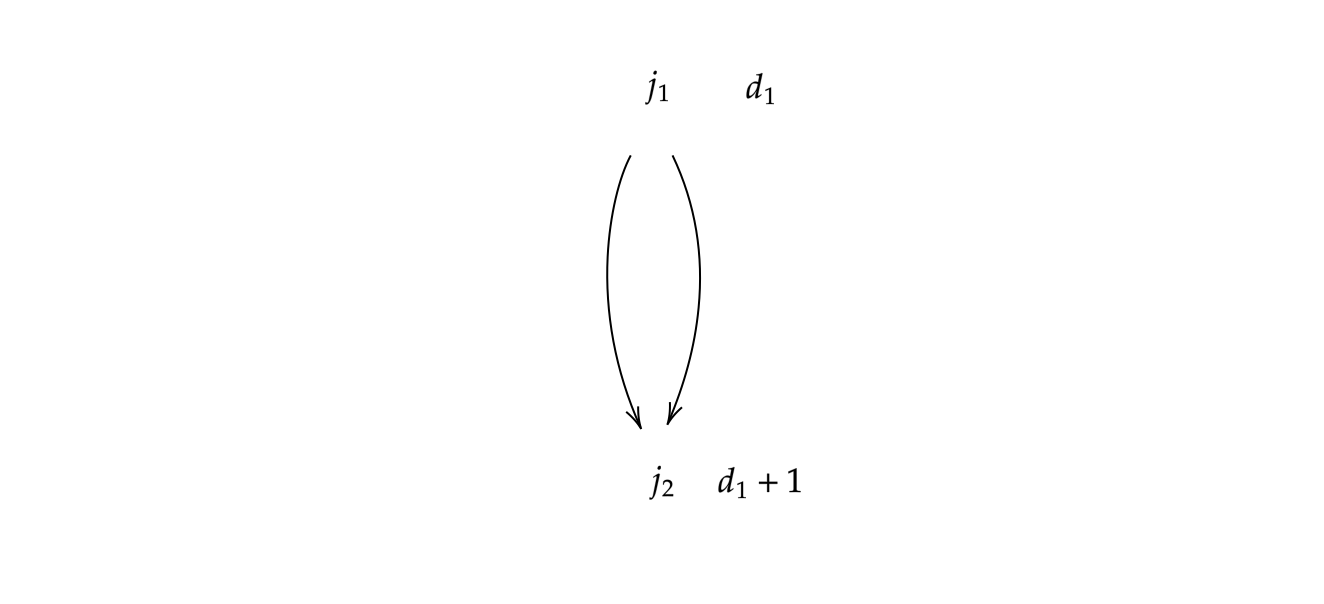}
\end{center}
where vertices $j_1$ and $j_2$ are decorated with dimensions $d_1$, $d_1+1$ respectively. This is known to be a point. Combining \eqref{eqn:defbps} and \eqref{eqn:bps}, we always have 
\[F_{d,d+1}^{\mathcal{O}_{F_2}(-C_2-f)}=(-q)^{1/2}/(q - 1).\]
If $d_2-d_1=2$, then $M_{\beta}^{\cO_X(-D)}$ is the moduli associated to following quiver:
\begin{center}
\includegraphics[width=0.9\textwidth]{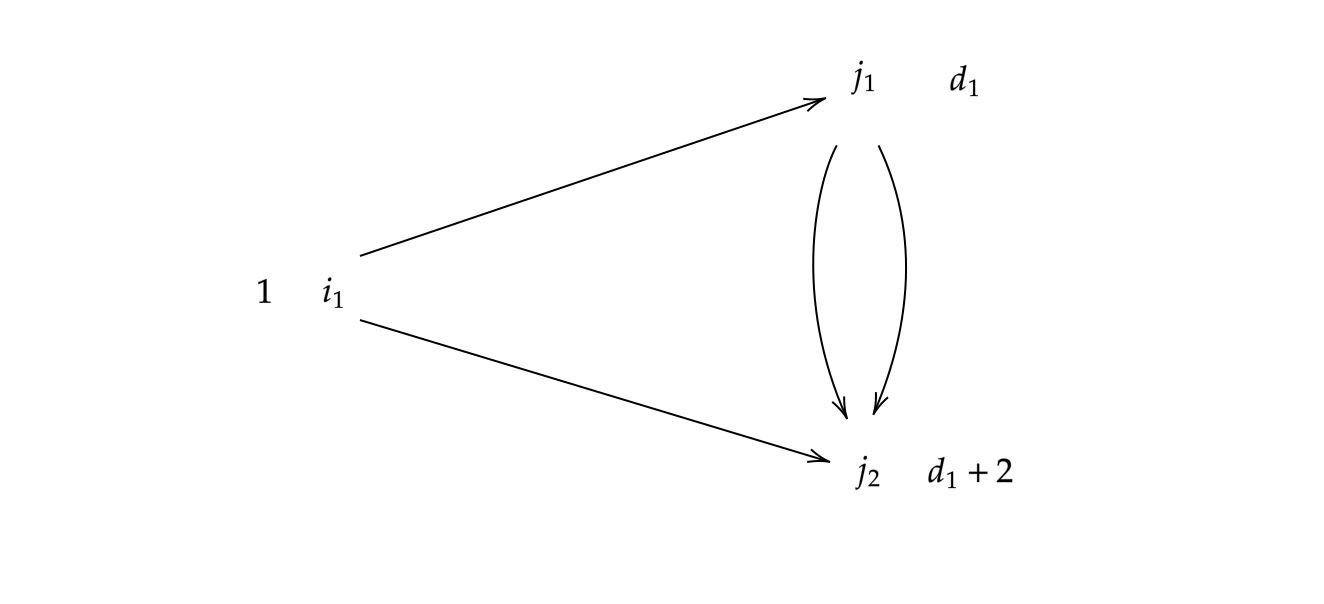}
\end{center}
The vertices $i_1$, $j_1$, $j_2$ are decorated with dimensions $1$, $d_1$, $d_1+2$ respectively. We can assume that $d_1>0$. Otherwise, $M_{\beta}^{\cO_X(-D)}$ is empty. Then by Theorem \ref{thm:dual}, $M_{\beta}^{\cO_X(-D)}$ is isomorphic to the moduli associated to the following quiver:
\begin{center}
\includegraphics[width=0.9\textwidth]{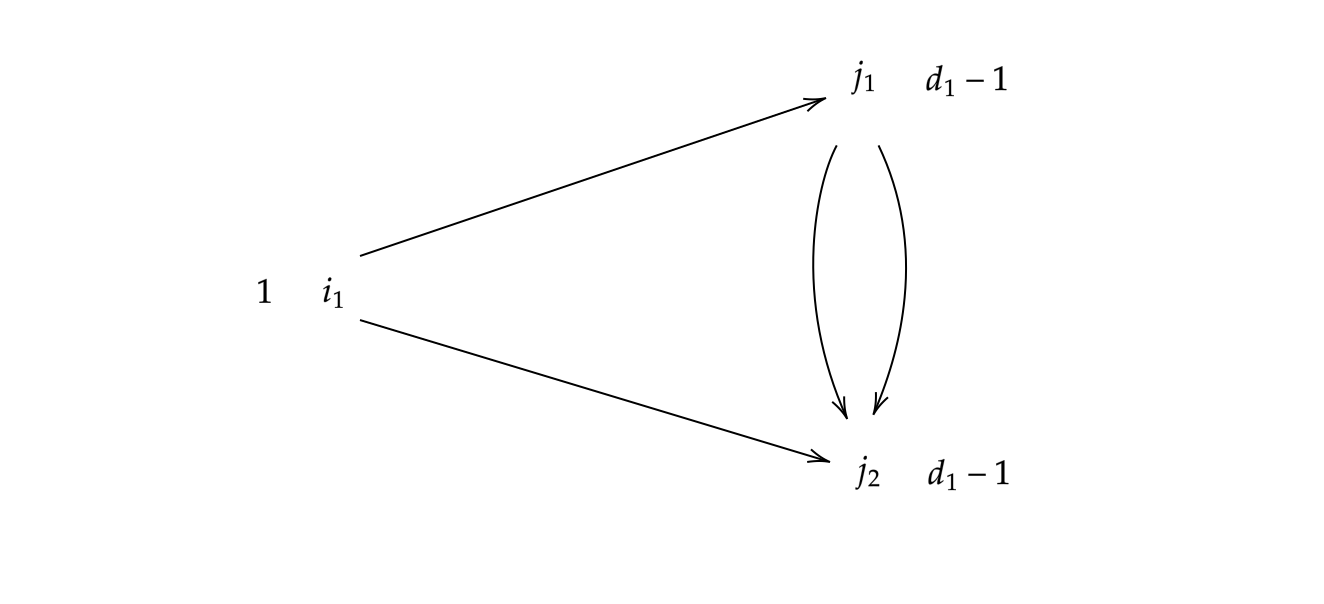}
\end{center}
The latter is a framed quiver moduli by Theorem \ref{thm:f/uf}. So the quiver DT invariants $\Omega_{M_{\beta}^{\cO_X(-D)}}(q)$ can be computed via \eqref{eqn:keyrl} above. More precisely, we take the quiver $Q$ to be the 2-Kronecker quiver with vertices $j_1$, $j_2$ and set $\underline{n}=e_{j_1}+e_{j_2}$ and $\theta=e_{j_1}^*-e_{j_2}^*$ in \eqref{eqn:keyrl}. Then $\Omega_{M_{\beta}^{\cO_X(-D)}}(q)$ in this case is simply the coefficient of $x^{\underline{d}}$ with $\underline{d}=(d_1-1)\underline{n}$. Combining with the facts that $M_{e_{j_1}+e_{j_2}}^{\theta-sst}=\PP^1$ and
$M_{d(e_{j_1}+e_{j_2})}^{\theta-st}=\emptyset$ for $d>1$, we may deduce that
\[\Omega_{M_{\beta}^{\cO_X(-D)}}(q)=\left[\frac{1}{(1-qx)(1-q^{-1}x)(1-2x)}\right]_{x^{d_1-1}}\]
where $[\cdot]_{x^{d_1-1}}$ extracts the coefficient of $x^{d_1-1}$. We have now computed $\Omega_{M_{\beta}^{\cO_X(-D)}}(q)$ for those $\beta$ such that $0<\Tg\cdot\beta<3$. Combining \eqref{eqn:defbps} and \eqref{eqn:bps}, we can determine all the initial $F_{\beta}^{\cO_X(-D)}$. Here is the table for $F_{d_1,d_2}^{\mathcal{O}_{F_2}(-C_2-f)}$:
\begin{table}[htb]  
\begin{center}   
\caption{GW-invariants of $\mathcal{O}_{F_2}(-C_2-f)$}
\begin{tabular}{|c|c|}  
\hline $(d_1,d_2)$ & $F_{d_1,d_2}^{\mathcal{O}_{F_2}(-C_2-f)}$\\
\hline $(d,d+1)$ & \footnotesize  $(-q)^{1/2}/(q - 1)$\\
\hline $(1,3)$ & \footnotesize  $-1$\\
\hline $(2,4)$ & \footnotesize  $-(q^2 + 2q + 1)/q$\\
\hline $(2,5)$ & \footnotesize  $(q-1)(q^4 + 3q^3 + 5q^2 + 3q + 1)/(-q)^{5/2}$\\
\hline $(2,6)$ & \footnotesize $(q-1)^2(q^6 + 4q^5 + 8q^4 + 12q^3 + 8q^2 + 4q + 1)/q^4$\\
\hline $(2,7)$ & \footnotesize  $(q - 1)^3(q^8 + 5q^7 + 12q^6 + 20q^5 + 28q^4 + 20q^3 + 12q^2 + 5q + 1)/(-q)^{11/2}$\\
\hline $(3,5)$ & \footnotesize $-(q^4 + 2q^3 + 5q^2 + 2q + 1)/q^2$    \\
\hline $(3,6)$ & \footnotesize $(q - 1)(q^8 + 3q^7 + 8q^6 + 14q^5 + 20q^4 + 14q^3 + 8q^2 + 3q + 1)/(-q)^{9/2}$ \\
\hline $(3,7)$ & \tabincell{c}{\footnotesize  $(q - 1)^2(q^{12} + 4q^{11} + 12q^{10} + 26q^9 + 49q^8 + 74q^7$\\ \footnotesize $ + 90q^6 + 74q^5 + 49q^4 + 26q^3 + 12q^2 + 4q + 1)/q^7$}\\
\hline $(4,6)$ & \footnotesize $ -(q^6 + 2q^5 + 5q^4 + 10q^3 + 5q^2 + 2q + 1)/q^3
 $\\
\hline $(4,7)$ & \tabincell{c}{\footnotesize  $-(q-1)(q^{12} + 3q^{11} + 8q^{10} + 17q^9 + 31q^8 + 47q^7$ \\ \footnotesize  $ + 60q^6 + 47q^5 + 31q^4 + 17q^3 + 8q^2 + 3q + 1)/(-q)^{13/2}$}\\
\hline $(5,7)$ & \footnotesize  $-(q^8 + 2q^7 + 5q^6 + 10q^5 + 21q^4 + 10q^3 + 5q^2 + 2q + 1)/q^4 $\\
\hline
\end{tabular}
\end{center}   
\end{table}

The BPS invariants $n_{g,(d_1,d_2)}^{\mathcal{O}_{F_2}(-C_2-f)}$ can be determined from the following formula
\[F_{d_1,d_2}^{\mathcal{O}_{F_2}(-C_2-f)}=\sum_{g\geq 0} n_{g,(d_1,d_2)}^{\mathcal{O}_{F_2}(-C_2-f)}(2\sin(h/2))^{2g-2+d_2-d_1}.\]

\begin{table}[htb]  
\begin{center}   
\caption{BPS invariants for $\mathcal{O}_{F_2}(-C_2-f)$}
\begin{tabular}{|l|c|c|c|c|c|c|c|}  
\hline 
\diagbox{$(d_1,d_2)$}{$g$} & $0$& $1$ &$2$ & $3$ & $4$ & $5$ & $6$  \\
\hline $(d,d+1)$ & $1$& $0$ &$0$ & $0$ & $0$ & $0$ & $0$  \\
\hline $(1,3)$ & $-1$ & $0$ & $0$ & $0$ & $0$ & $0$ & $0$  \\
\hline $(2,4)$ & $-4$ & $1$ & $0$ & $0$ & $0$ & $0$ & $0$   \\
\hline $(2,5)$ & $13$ & $-7$ & $1$ & $0$ & $0$ & $0$ & $0$  \\
\hline $(2,6)$ & $-38$ & $33$ & $-10$ & $1$ & $0$ & $0$ & $0$  \\
\hline $(2,7)$ & $104$ & $-129$ & $62$ & $-13$ & $1$ & $0$ & $0$  \\
\hline $(3,5)$ & $-11$ & $6$ & $-1$ & $0$ & $0$ & $0$ & $0$   \\
\hline $(3,6)$ & $72$ & $-89$ & $46$ & $-11$ & $1$ & $0$ & $0$ \\
\hline $(3,7)$ & $-422$ & $832$ & $-750$ & $374$ & $-106$ & $16$ & $-1$  \\
\hline $(4,6)$ & $-26$ & $22$ & $-8$ & $1$ & $0$ & $0$ & $0$  \\
\hline $(4,7)$ & $274$ & $-563$ & $548$ & $-298$ & $92$ & $-15$ & $1$  \\
\hline $(5,7)$ & $-57$ & $64$ & $-37$ & $10$ & $-1$ & $0$ & $0$  \\
\hline 
\end{tabular}
\end{center}   
\end{table}

In general, in order to determine those Gromov--Witten invariants of $\mathcal{O}_{F_n}(-C_n-f)$ with $n>2$, we need to determine all the quiver DT invariants for the following two types of quivers:
\begin{center}
\includegraphics[width=0.9\textwidth]{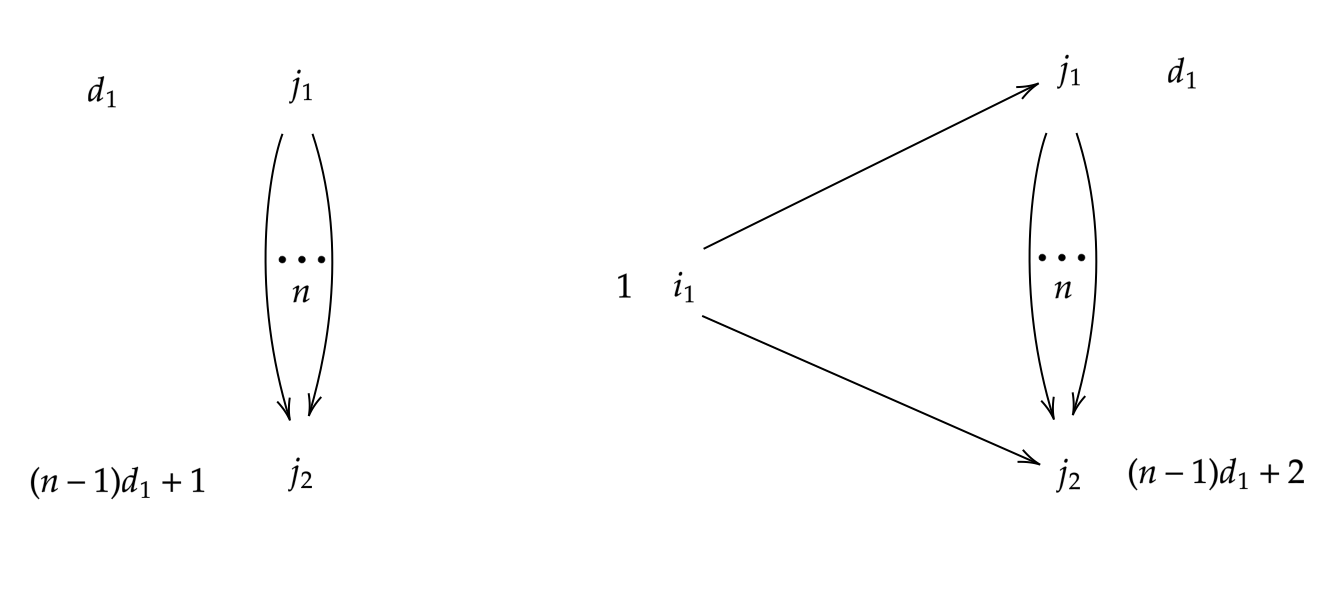}
\end{center}
where there are $n$ arrows between vertices $j_1$ and $j_2$ in both quivers and dimensions associated to different vertices are indicated in the picture. As far as we know, the quiver DT invariants for $n$-Kronecker quiver with $n>2$ is quite hard to compute. We do not know how to completely solve the quiver DT invariants for the above two types of quivers.

\appendix
\section{A Counterexample}\label{sec:ex}
In this appendix, we give a counterexample to the genus-zero recursion \eqref{eqn:g0loc} when $D$ is only assumed to be nef. This shows in particular that  Theorem \ref{thm:main} may fail once we relax the ample condition of $D$ to be nef. 

Let $X$ be the Hirzebruch surface $F_1=\mathbb{P}(\mathcal{O}(1)\oplus \mathcal{O})$ and $D$ be a fiber of $F_1$. We have the following recursion if $\Tg\cdot\beta\geq 3$:
\begin{equation*}
N_{0,\beta}^{\cO_X(-D)}=-\sum_{\beta_1+\beta_2=\beta\atop \beta_1,\beta_2>0} N_{0,\beta_1}^{\cO_X(-D)}N_{0,\beta_2}^{\cO_X(-D)}\left(D\cdot\beta_1\right)^2{\Tg\cdot\beta-3\choose \Tg\cdot\beta_1-1}+(D\cdot\beta)^2N_{0,\beta-f}^{\cO_X(-D)}
\end{equation*}
where $f$ is the fiber class and $\beta$ is not a multiple of the fiber class $f$. The above recursion can be deduced similarly as in Section \ref{sec:recg0}. First, using WDVV equation, we can deduce the following recursion for the genus-zero relative invariants of the above pair $(X,D)$ as in \cite{FW}:
\begin{equation}\label{eqn:relrecur}
\frac{N_{0,\beta}^{X/D}}{D\cdot\beta}=\sum_{\beta_1+\beta_2=\beta\atop \beta_1,\beta_2>0} \frac{N_{0,\beta_1}^{X/D}}{D\cdot\beta_1}\frac{N_{0,\beta_2}^{X/D}}{D\cdot\beta_2}\left(D\cdot\beta_1\right)^2{\Tg\cdot\beta-3\choose \Tg\cdot\beta_1-1}+(D\cdot\beta)N_{0,\beta-f}^{X/D}
\end{equation}
where $\Tg\cdot\beta\geq 3$. The appearance of the additional term $(D\cdot\beta)N_{0,\beta-f}^{X/D}$ can be illustrated via the following example.

Let $C_1$ be the section of $F_1$ with self intersection number $1$. Then the above recursion implies that 
\begin{equation}\label{eqn:illexample}
N_{0,C_1+nf}^{X/D}=N_{0,C_1}^{X/D},\,\forall n>0.
\end{equation}
Actually, in this case the contact order equals to $D\cdot (C_1+nf)=1$, and so the relative invariant $N_{0,C_1+nf}^{X/D}$ equals to the corresponding absolute invariant $N_{0,C_1+nf}^X$ which counts rational curves in $X$ passing through $2n+2$ generic points. By \cite[Example 8.1]{Ga2}, we know that $N_{0,C_1+nf}^X=1$ for all $n\geq 0$. This confirms the identity \eqref{eqn:illexample}.

In order to deduce the recursion for $N_{0,\beta}^{\cO_X(-D)}$, we again need the following local/relative correspondence
\begin{equation}\label{eqn:loc/rel2}
N_{0,\beta}^{X/D}=(-1)^{D\cdot\beta-1}(D\cdot\beta)N_{0,\beta}^{\cO_X(-D)}.
\end{equation}
According to \cite{GGR}, the above equality holds with the extra condition that 
$D\cdot\beta>0$. But it is easy to compute that $N_{0,\beta}^{X/D}=0$ if $\beta$ is some multiple of the fiber class $f$. So in the recursion \eqref{eqn:relrecur}, we can always assume that $D\cdot\beta_i>0$, $i=1,2$. The recursion for $N_{0,\beta}^{\cO_X(-D)}$ then follows from \eqref{eqn:relrecur} and \eqref{eqn:loc/rel2}.

\section{Genus $0$ recursion}\label{sec:recg0} 
In this section, we give a direct proof in Gromov--Witten theory of the recursion formula for genus-zero local invariants $N_{0,\beta}^{\cO_X(-D)}$:
\begin{equation}\label{eqn:g0loc}
N_{0,\beta}^{\cO_X(-D)}=-\sum_{\beta_1+\beta_2=\beta\atop \beta_1,\beta_2>0} N_{0,\beta_1}^{\cO_X(-D)}N_{0,\beta_2}^{\cO_X(-D)}\left(D\cdot\beta_1\right)^2{\Tg\cdot\beta-3\choose \Tg\cdot\beta_1-1}.
\end{equation}

Let $N_{0,\beta}^{X/D}$ be the (virtual) number of rational curves in $X$ with curve class $\beta$ and passing through $\Tg\cdot\beta$ general points and meet $D$ at an unspecified point with contact order $D\cdot\beta$. According to \cite[Theorem 1.1]{FW}, we have the following recursion formula when $\Tg\cdot\beta\geq 3$:
\begin{equation*}
    \begin{split}
        (H\cdot D)\frac{N_{0,\beta}^{X/D}}{d}=& 
        \sum\limits_{\beta_1+\beta_2=\beta\atop \beta_1,\beta_2>0} \left(d_1^2(H\cdot\beta_2){\Tg\cdot\beta-3\choose \Tg\cdot\beta_1 -1}+d_1d_2(H\cdot\beta_2){\Tg\cdot\beta-3\choose \Tg\cdot\beta_1-2} \right.\\
        &-\left. d_1^2(H\cdot\beta_1){\Tg\cdot\beta-3 \choose \Tg\cdot\beta_1} - d_1d_2(H\cdot \beta_1){\Tg\cdot\beta-3 \choose \Tg\cdot\beta_1-1}\right)\frac{N_{0,\beta_1}^{X/D}}{d_1}\frac{N_{0,\beta_2}^{X/D}}{d_2}
    \end{split}
\end{equation*}
where $H$ is any divisor, $d=D\cdot\beta$, $d_1=D\cdot\beta_1$, $d_2=D\cdot\beta_2$. After replacing $H$ by $\Tg$, the above recursion formula can be simplified as 
\begin{equation}\label{eqn:relg0}
(\Tg\cdot D)\frac{N_{0,\beta}^{X/D}}{d}=2\sum\limits_{\beta_1+\beta_2=\beta\atop \beta_1,\beta_2>0}d_1^2{\Tg\cdot\beta-3\choose \Tg\cdot\beta_1-1}\frac{N_{0,\beta_1}^{X/D}}{d_1}\frac{N_{0,\beta_2}^{X/D}}{d_2}\,.
\end{equation}
This follows from the combinatorial identities
\[(\Tg\cdot\beta_2){\Tg\cdot\beta-3\choose \Tg\cdot\beta_1 -1}-(\Tg\cdot\beta_1){\Tg\cdot\beta-3 \choose \Tg\cdot\beta_1}=2{\Tg\cdot\beta-3\choose \Tg\cdot\beta_1-1}\]
and 
\[\sum\limits_{\beta_1+\beta_2=\beta\atop \beta_1,\beta_2>0}\left((\Tg\cdot\beta_2){\Tg\cdot\beta-3\choose \Tg\cdot\beta_1-2}-(\Tg\cdot\beta_1){\Tg\cdot\beta-3 \choose \Tg\cdot\beta_1-1}\right)N_{0,\beta_1}^{X/D}N_{0,\beta_2}^{X/D}=0. \]
Since $D$ is smooth and rational, we have $\Tg\cdot D=-(K_X+D)\cdot D=2$ by the adjunction formula. So \eqref{eqn:relg0} can be further simplified as
\begin{equation*}
\frac{N_{0,\beta}^{X/D}}{D\cdot\beta}=\sum\limits_{\beta_1+\beta_2=\beta\atop \beta_1,\beta_2>0}(D\cdot\beta_1)^2{\Tg\cdot\beta-3\choose \Tg\cdot\beta_1-1}\frac{N_{0,\beta_1}^{X/D}}{D\cdot\beta_1}\frac{N_{0,\beta_2}^{X/D}}{D\cdot\beta_2}\,.
\end{equation*}
Now \eqref{eqn:g0loc} follows from the local/relative correspondence \cite{GGR}:
\[N_{0,\beta}^{X/D}=(-1)^{D\cdot\beta-1}(D\cdot\beta)N_{0,\beta}^{\cO_X(-D)}.\]
From the proof, it is clear that the requirement of $D$ to be rational, i.e., $\Tg\cdot D=2$ is necessary.

\section{Comparison with genus-one recursion from the Virasoro constraints} \label{app:virasoro}
From the recursion \eqref{eqn:recurfm}, we obtain the following recursion for genus one Gromov--Witten invariants of $\mathcal{O}_{\mathbb{P}^2}(-1)$:
\[\begin{aligned}
N_{1,d}=\sum_{d_1+d_2=d\atop d_1,d_2>0}\left(N_{0,d_1}N_{0,d_2}\frac{d_1^4}{12}-(N_{0,d_1}N_{1,d_2}+N_{0,d_2}N_{1,d_1})d_1^2\right){2d-3\choose 2d_1-1}\,.
\end{aligned}\]

We may also embed $\mathcal{O}_{\mathbb{P}^2}(-1)$ into $\mathbb{P}(\mathcal{O}_{\mathbb{P}^2}(-1)\oplus \mathcal{O}_{\mathbb{P}^2})$. Then Gromov--Witten invariants of $\mathcal{O}_{\mathbb{P}^2}(-1)$ equal to the corresponding invariants of $\mathbb{P}(\mathcal{O}_{\mathbb{P}^2}(-1)\oplus \mathcal{O}_{\mathbb{P}^2})$. Since $\mathbb{P}(\mathcal{O}_{\mathbb{P}^2}(-1)\oplus \mathcal{O}_{\mathbb{P}^2})$ is a projective toric variety, we can apply the Virasoro constraints \cite{EHX, EJX, Giv5, viraosorotoric, T12},  and get another recursion:
\[N_{1,d}=-\frac{d(d-1)}{24}N_{0,d}-\sum_{d_1+d_2=d\atop d_1,d_2>0}\frac{(d-1)(2d-1)}{2}{2d-3\choose 2d_1-2}N_{0,d_1}N_{1,d_2}\,.\]
By Theorem \ref{thm:main} and the Virasoro constraints, these two recursions are equivalent.
Using computer, we checked directly that up to degree 19, it is indeed the case. However, we do not know a direct elementary proof of the equivalence of these two recursions.

\bibliography{universal-BIB}

\begin{bibdiv}
\begin{biblist}

\bib{ABflow}{article}{
      author={Arg\"{u}z, H.},
      author={Bousseau, P.},
       title={The flow tree formula for {D}onaldson-{T}homas invariants of
  quivers with potentials},
        date={2022},
     journal={Compos. Math.},
      volume={158},
      number={12},
       pages={2206\ndash 2249},
}

\bib{APflow}{article}{
      author={Alexandrov, S.},
      author={Pioline, B.},
       title={Attractor flow trees, {BPS} indices and quivers},
        date={2019},
        ISSN={1095-0761},
     journal={Adv. Theor. Math. Phys.},
      volume={23},
      number={3},
       pages={627\ndash 699},
}

\bib{BFGW}{article}{
      author={Bousseau, P.},
      author={Fan, H.},
      author={Guo, S.},
      author={Wu, L.},
       title={Holomorphic anomaly equation for {$(\mathbb{P}^2,E)$} and the
  {N}ekrasov-{S}hatashvili limit of local {$\mathbb{P}^2$}},
        date={2021},
     journal={Forum Math. Pi},
      volume={9},
       pages={Paper No. e3, 57},
}

\bib{BGP}{article}{
      author={Bernstein, I.N.},
      author={Gelfand, I.M.},
      author={Ponomarev, V.A.},
       title={Coxeter functors, and {G}abriel's theorem},
        date={1973},
     journal={Uspehi Mat. Nauk},
      volume={28},
      number={2(170)},
       pages={19\ndash 33},
}

\bib{Bou20}{article}{
      author={Bousseau, P.},
       title={The quantum tropical vertex},
        date={2020},
     journal={Geom. Topol.},
      volume={24},
      number={3},
       pages={1297\ndash 1379},
}

\bib{Bou21}{article}{
      author={Bousseau, P.},
       title={On an example of quiver {D}onaldson-{T}homas/relative
  {G}romov-{W}itten correspondence},
        date={2021},
     journal={Int. Math. Res. Not. IMRN},
      number={15},
       pages={11845\ndash 11888},
}

\bib{CI}{article}{
      author={{Coates}, T.},
      author={{Iritani}, H.},
       title={Gromov-{W}itten invariants of local {$\Bbb P^2$} and modular
  forms},
        date={2021},
     journal={Kyoto J. Math.},
      volume={61},
      number={3},
       pages={543\ndash 706},
}

\bib{DIW}{article}{
      author={Doan, A.},
      author={Ionel, E.-N},
      author={Walpuski, T.},
       title={The {G}opakumar-{V}afa finiteness conjecture},
        date={2021},
     journal={arXiv preprint arXiv:2103.08221},
}

\bib{DW19}{article}{
      author={Doan, A.},
      author={Walpuski, T.},
       title={Counting embedded curves in symplectic 6-manifolds},
        date={2019},
     journal={arXiv preprint arXiv:1910.12338},
}

\bib{EHX}{article}{
      author={{Eguchi}, T.},
      author={{Hori}, K.},
      author={{Xiong}, C.-S.},
       title={{Quantum cohomology and Virasoro algebra}},
        date={1997},
     journal={Phys. Lett. B},
      volume={402},
       pages={71\ndash 80},
}

\bib{EJX}{article}{
      author={{Eguchi}, T.},
      author={{Jinzenji}, M.},
      author={{Xiong}, C.-S.},
       title={{Quantum cohomology and free field representations}},
        date={1998},
     journal={Nuclear Phys. B},
      volume={510},
       pages={608\ndash 622},
}

\bib{ER}{article}{
      author={Engel, J.},
      author={Reineke, M.},
       title={Smooth models of quiver moduli},
        date={2009},
     journal={Math. Z.},
      volume={262},
      number={4},
       pages={817\ndash 848},
}

\bib{FP}{article}{
      author={{Faber}, C.},
      author={{Pandharipande}, R.},
       title={{Hodge integrals and Gromov--Witten theory}},
        date={2000},
     journal={Invent. Math.},
      volume={139},
      number={1},
       pages={173\ndash 199},
}

\bib{FRZZ}{article}{
      author={Fang, B.},
      author={Ruan, Y.},
      author={Zhang, Y.},
      author={Zhou, J.},
       title={Open {G}romov-{W}itten theory of {$K_{\mathbb{P}^2}$},
  {$K_{\mathbb{P}^1\times\mathbb{ P^1}}$}, {$K_{W\mathbb{P}[1,1,2]}$},
  {$K_{\mathbb{F}_1}$} and {J}acobi forms},
        date={2019},
     journal={Comm. Math. Phys.},
      volume={369},
      number={2},
       pages={675\ndash 719},
}

\bib{FW}{article}{
      author={Fan, H.},
      author={Wu, L.},
       title={Witten-{D}ijkgraaf-{V}erlinde-{V}erlinde equation and its
  application to relative {G}romov-{W}itten theory},
        date={2021},
     journal={Int. Math. Res. Not. IMRN},
      number={13},
       pages={9834\ndash 9852},
}

\bib{Ga2}{article}{
      author={Gathmann, A.},
       title={Gromov-{W}itten invariants of blow-ups},
        date={2001},
     journal={J. Algebraic Geom.},
      volume={10},
      number={3},
       pages={399\ndash 432},
}

\bib{GGR}{article}{
      author={{Garrel}, M.~van},
      author={{Graber}, T.},
      author={{Ruddat}, H.},
       title={Local {G}romov-{W}itten invariants are log invariants},
        date={2019},
     journal={Adv. Math.},
      volume={350},
       pages={860\ndash 876},
}

\bib{Giv5}{article}{
      author={Givental, A.},
       title={{Gromov--Witten invariants and quantization of quadratic
  hamiltonians}},
        date={2001},
     journal={Moscow Mathematical Journal},
      volume={1},
      number={5},
       pages={551\ndash 568},
}

\bib{GrP}{article}{
      author={Gross, M.},
      author={Pandharipande, R.},
       title={Quivers, curves, and the tropical vertex},
        date={2010},
     journal={Port. Math.},
      volume={67},
      number={2},
       pages={211\ndash 259},
}

\bib{IP2}{article}{
      author={Ionel, E.-N.},
      author={Parker, T.-H.},
       title={The {G}opakumar-{V}afa formula for symplectic manifolds},
        date={2018},
     journal={Ann. of Math. (2)},
      volume={187},
      number={1},
       pages={1\ndash 64},
}

\bib{viraosorotoric}{article}{
      author={Iritani, H.},
       title={Convergence of quantum cohomology by quantum {L}efschetz},
        date={2007},
        ISSN={0075-4102},
     journal={J. Reine Angew. Math.},
      volume={610},
       pages={29\ndash 69},
}

\bib{JS}{article}{
      author={Joyce, D.},
      author={Song, Y.},
       title={A theory of generalized {D}onaldson-{T}homas invariants},
        date={2012},
        ISSN={0065-9266},
     journal={Mem. Amer. Math. Soc.},
      volume={217},
      number={1020},
       pages={iv+199},
}

\bib{KS}{article}{
      author={Kontsevich, M.},
      author={Soibelman, Y.},
       title={Stability structures, motivic {D}onaldson-{T}homas invariants and
  cluster transformations},
        date={2008},
     journal={arXiv preprint arXiv:0811.2435},
}

\bib{Lho}{article}{
      author={Lho, H.},
       title={Gromov-{W}itten invariants of {C}alabi-{Y}au manifolds with two
  {K}\"{a}hler parameters},
        date={2021},
     journal={Int. Math. Res. Not. IMRN},
      number={10},
       pages={7552\ndash 7596},
}

\bib{LhoP}{article}{
      author={{Lho}, H.},
      author={{Pandharipande}, R.},
       title={Stable quotients and the holomorphic anomaly equation},
        date={2018},
     journal={Adv. Math.},
      volume={332},
       pages={349\ndash 402},
}

\bib{LP}{article}{
      author={Lanteri, A.},
      author={Palleschi, M.},
       title={About the adjunction process for polarized algebraic surfaces},
        date={1984},
     journal={J. Reine Angew. Math.},
      volume={352},
       pages={15\ndash 23},
}

\bib{LTvir}{article}{
      author={Liu, X.},
      author={Tian, G.},
       title={Virasoro constraints for quantum cohomology},
        date={1998},
     journal={J. Differential Geom.},
      volume={50},
      number={3},
       pages={537\ndash 590},
}

\bib{MPattr}{article}{
      author={Mozgovoy, S.},
      author={Pioline, B.},
       title={Attractor invariants, brane tilings and crystals},
        date={2020},
     journal={arXiv preprint arXiv:2012.14358},
}

\bib{MR}{article}{
      author={Meinhardt, S.},
      author={Reineke, M.},
       title={Donaldson-{T}homas invariants versus intersection cohomology of
  quiver moduli},
        date={2019},
     journal={J. Reine Angew. Math.},
      volume={754},
       pages={143\ndash 178},
}

\bib{R08}{incollection}{
      author={Reineke, M.},
       title={Moduli of representations of quivers},
        date={2008},
   booktitle={Trends in representation theory of algebras and related topics},
      series={EMS Ser. Congr. Rep.},
   publisher={Eur. Math. Soc., Z\"{u}rich},
       pages={589\ndash 637},
}

\bib{R10}{article}{
      author={Reineke, M.},
       title={Poisson automorphisms and quiver moduli},
        date={2010},
     journal={J. Inst. Math. Jussieu},
      volume={9},
      number={3},
       pages={653\ndash 667},
}

\bib{R17}{incollection}{
      author={Reineke, M.},
       title={Quiver moduli and small desingularizations of some {GIT}
  quotients},
        date={2017},
   booktitle={Representation theory---current trends and perspectives},
      series={EMS Ser. Congr. Rep.},
   publisher={Eur. Math. Soc., Z\"{u}rich},
       pages={613\ndash 635},
}

\bib{RSW}{article}{
      author={Reineke, M.},
      author={Stoppa, J.},
      author={Weist, T.},
       title={M{PS} degeneration formula for quiver moduli and refined
  {GW}/{K}ronecker correspondence},
        date={2012},
     journal={Geom. Topol.},
      volume={16},
      number={4},
       pages={2097\ndash 2134},
}

\bib{RW13}{article}{
      author={Reineke, M.},
      author={Weist, T.},
       title={Refined {GW}/{K}ronecker correspondence},
        date={2013},
     journal={Math. Ann.},
      volume={355},
      number={1},
       pages={17\ndash 56},
}

\bib{RW21}{article}{
      author={Reineke, M.},
      author={Weist, T.},
       title={Moduli spaces of point configurations and plane curve counts},
        date={2021},
     journal={Int. Math. Res. Not. IMRN},
      number={13},
       pages={10339\ndash 10372},
}

\bib{T12}{article}{
      author={Teleman, C.},
       title={The structure of 2{D} semi-simple field theories},
        date={2012},
        ISSN={0020-9910},
     journal={Invent. Math.},
      volume={188},
      number={3},
       pages={525\ndash 588},
}

\bib{Wan2}{article}{
      author={Wang, X.},
       title={{F}inite generation and holomorphic anomaly equation for
  equivariant {G}romov-{W}itten invariants of {$K_{\mathbb{P}^1 \times
  \mathbb{P}^1}$}},
        date={2019},
     journal={arXiv preprint arXiv:1908.03691},
}

\bib{W13}{article}{
      author={Weist, T.},
       title={Localization in quiver moduli spaces},
        date={2013},
     journal={Represent. Theory},
      volume={17},
       pages={382\ndash 425},
}

\bib{Z}{article}{
      author={Zinger, A.},
       title={A comparison theorem for {G}romov-{W}itten invariants in the
  symplectic category},
        date={2011},
     journal={Adv. Math.},
      volume={228},
      number={1},
       pages={535\ndash 574},
}

\end{biblist}
\end{bibdiv}
\bibliographystyle{amsxport}

\end{document}